\documentclass[11pt]{article}
\usepackage{color}
\usepackage{relsize}
\usepackage{amsmath,amsthm,amssymb}
\usepackage{amsfonts}
\usepackage{enumerate}
\usepackage{appendix}

\usepackage{graphicx}

\newtheorem{thm}{Theorem}[section]
\newtheorem{cor}[thm]{Corollary}
\newtheorem{lem}[thm]{Lemma}
\newtheorem{prop}[thm]{Proposition}

\theoremstyle{definition}
\newtheorem{assu}[thm]{Assumption}
\newtheorem{defn}[thm]{Definition}

\numberwithin{equation}{section}
\def\bR{\mathbb{R}}

\def\bN{\mathbb{N}}

\theoremstyle{remark}
\newtheorem{rem}[thm]{Remark}
\begin{document}
	
	\title{Energy Transfer, Weak Resonance, and Fermi's Golden Rule in Hamiltonian Nonlinear Klein-Gordon Equations}
	
	\author{Zhen Lei \footnotemark[1]\ \footnotemark[2]%\ \footnotemark[2]
		\and Jie Liu    \footnotemark[1]\ \footnotemark[3]% \ \footnotemark[3]
		\and Zhaojie Yang  \footnotemark[1]\ \footnotemark[4]%\ \footnotemark[4]
	}
	\renewcommand{\thefootnote}{\fnsymbol{footnote}}
	\footnotetext[1]{School of Mathematical Sciences; LMNS and Shanghai Key Laboratory for Contemporary Applied Mathematics, Fudan University, Shanghai 200433, P. R.China.} \footnotetext[2]{Email: zlei@fudan.edu.cn}
	\footnotetext[3]{Email: j\_liu18@fudan.edu.cn}
	\footnotetext[4]{Email: yangzj20@fudan.edu.cn}
	
	\date{\today}
	
	\maketitle
	
	\begin{abstract}
		This paper focuses on a class of nonlinear Klein-Gordon equations in three dimensions, which are Hamiltonian perturbations of the linear Klein-Gordon equation with potential. The unperturbed dynamical system has a bound state with frequency $\omega$, a spatially localized and time periodic solution. In quantum mechanics, metastable states, which last longer than expected, have been observed. These metastable states are a consequence of the instability of the bound state under the nonlinear Fermi's Golden Rule. 
		
		In this study, we explore the underlying mathematical instability mechanism from the bound state to these metastable states. Besides, we derive the sharp energy transfer rate from discrete to continuum modes, when the discrete spectrum was not close to the continuous spectrum of the Sch\"ordinger operator $H= -\Delta + V + m^2$, i.e. weak resonance regime $ \sigma_c(\sqrt{H}) = [m, \infty)$, $0< 3\omega < m$. This extends the work of Soffer and Weinstein \cite{SW1999} for resonance regime $3\omega > m$ and confirms their conjecture in \cite{SW1999}. 
		Our proof relies on a more refined version of normal form transformation of Bambusi and Cuccagna \cite{BC}, the generalized Fermi's Golden Rule, as well as certain weighted dispersive estimates.

		%In 1999, Soffer and Weinstein began exploring  in their work \cite{SW1999}. They found that the rate of energy transfer from discrete to continuum modes was anomalously slow at $1/(1 + t)^{\frac{1}{4}}$, when the discrete spectrum was close to the continuous spectrum . Since then, many efforts have been focused on the mechanism and the rate of energy transfer when $0 < 3\omega < m$, where the discrete spectrum is farther from the continuous spectrum of $H$.
		
		%the decay mechanism of metastable states by employing  in the weak resonance regime $0 < 3\omega < m$, which was conjectured by Soffer and Weinstein \cite{SW1999}. , we obtain . This confirms 
	\end{abstract}
	
	\tableofcontents

	\section{Introduction}
	\subsection{Background}
	The well-known Kolmogorov-Arnold-Moser (KAM) theory is concerned with the persistence of periodic and quasi-periodic motion under the Hamiltonian perturbation of a dynamical system, which has many applications in various research fields, such as the N-body problem in celestial mechanics  \cite{K,L}. For a finite dimensional integrable Hamiltonian system, the aforementioned theory was initiated by Kolmogorov \cite{K} and then extended by Moser \cite{M} and Arnold \cite{A}. Subsequently, many efforts have been focused on generalizing the KAM theory to infinite dimensional Hamiltonian systems (Hamiltonian PDEs), wherein solutions are defined on compact spatial domains, such as \cite{Bo,CW,Ku}. In all of the above results, appropriate non-resonance conditions imply the persistence of periodic and quasi-periodic solutions. See \cite{L,W} and the references therein for a comprehensive survey.
	
	In this study, we focus on a different phenomenon that resonance conditions lead to the instability of periodic or quasi-periodic solutions. For example, in quantum mechanics, under a small perturbation, an excited state could be unstable with energy shifting to the ground state, free waves and nearby excited states. In this process anomalously long-lived states called metastable states in physics literature were observed. To study this instability mechanism, Dirac, in 1927, considered an unperturbed Hamiltonian $H_0$ with the initial eigenstate $i(x)$, and the perturbing Hamiltonian $H_1$ with the final eigenstate $f(x)$, then calculated the transition probability per unit time from the initial eigenstate to the final eigenstate 
	\begin{align}\label{Dirac}
		\Gamma_{i \rightarrow f}=\frac{2 \pi}{\hbar}\left|\int_{\mathbb{R}^{3}} i(x) H_1(x) f(x) d x\right|^{2} \cdot \rho_{f},
	\end{align}
	where $\rho_{f}$ is the density of the final states. This was utilized by Fermi in 1934 to establish his famous theory of beta decay, where he called the formula \eqref{Dirac} ``golden rule No. 2,'' thus now known as ``Fermi's Golden Rule."
	
	%	\subsection{The Work of Soffer-Weinstein}
	The rigorous mathematical analysis of the instability mechanism of bound states and energy transfer rates were acquired quite late and remain largely open. The first significant result of such problem for a nonlinear PDE was provided by Soffer and Weinstein \cite{SW1999}. Consider the Cauchy problem for a nonlinear Klein-Gordon equation with potential in the whole space,
	\begin{equation}\label{NLKG}
		\begin{cases}
			\partial^2_t u-\Delta u + m^2u + V(x) u = \lambda u^3, & t>0, x\in \mathbb{R}^3, \lambda\in \mathbb{R},\\
			u(x,0) = u_0(x), \quad \partial_t u(x,0) = u_1(x).
		\end{cases} 
	\end{equation}
	This equation is a Hamiltonian system with energy 
	\begin{align*}
		\mathcal{E}[u,\partial_{t} u]\equiv \frac12\int (\partial_t u)^2 + |\nabla u|^2 + m^2 u^2 + V(x)u^2 dx - \frac{\lambda}{4}  \int u^4 dx.
	\end{align*}
	The nonlinearity of $u^3$ can certainly be replaced by a general Hamiltonian one. See \cite{SW1999} for more details. Assume that the potential function $V(x)$ satisfies the following properties:
	\begin{assu}\label{assumption-V}
		Let $V (x)$ be real-valued function such that\\
		(V1) for $\delta > 5$ and $|\alpha| \le 2, |\partial^\alpha V (x)| \le  C_{\alpha} \langle x\rangle^{-\delta}$.\\
		(V2) zero is not a resonance of the operator $-\Delta + V$.\\
		(V3) $(-\Delta + 1)^{-1} \left((x\cdot\nabla)^l V(x)\right) (-\Delta + 1)^{-1}$ is bounded on $L^2$ for $|l| \le N_{\ast}$ with $N_{\ast} \ge 10$.\\
		(V4) the operator
		$H = -\Delta + V (x) + m^2 \triangleq B^2$
		has a continuous spectrum, $\sigma_{c}(H) = [m^2, +\infty)$, and a unique strictly positive
		simple eigenvalue, $\omega^2<m^2$ with the associated normalized eigenfunction $\varphi(x)$:
		$$B^2\varphi =\omega^2\varphi.$$
	\end{assu}
	Under these assumptions, one can observe that the linear Klein-Gordon equation \eqref{NLKG}, with $\lambda=0$, has a two-parameter family of spatially localized and time-periodic solutions of the form:
	\begin{align*}
		u(t,x) = R \cos(\omega t+\theta)\varphi(x).
	\end{align*}
	For $\lambda\neq 0$, i.e., Hamiltonian nonlinear perturbation of the linear dispersive equation, instead of the KAM type results, Soffer and Weinstein \cite{SW1999} proved that if $3\omega > m$ and the following resonance condition (analogous Fermi's Golden Rule Condition) was met,
	\begin{equation}\label{FGR-SW}
		\Gamma\equiv \frac{\pi}{3\omega}\left(\mathbf{P}_c\varphi^3(x), \delta(B-3\omega)\mathbf{P}_c\varphi^3(x)\right) \equiv \frac{\pi}{3\omega} |(\mathcal{F}_c\varphi^3)(3\omega)|^2>0.
	\end{equation}
	where $\mathbf{P}_c$ denotes the projection onto the continuous spectral part of $B$ and $\mathcal{F}_c$ denotes the Fourier transform relative to the continuous spectral part of $B$, then small global solutions to \eqref{NLKG} decayed to zero at an anomalously slow rate as time tended to infinity. More precisely, the solution $u(t,x)$ had the following expansion as $t\to \pm\infty$:
	\begin{align*}
		u(t,x) = R(t) \cos(\omega t+\theta(t))\varphi(x) +\eta(t,x),
	\end{align*}
	where
	\begin{align*}
		R(t) = \mathcal{O}(|t|^{-\frac 14}), \theta(t) = \mathcal{O}(|t|^{\frac 12}) \text{~and~} \quad  \|\eta(t,\cdot)\|_{L^8} = \mathcal{O}(|t|^{-\frac 34}).
	\end{align*}
	This result provides the exact rate of energy transfer from discrete to continuum modes. As a corollary, there are no small global periodic or quasi-periodic solutions to \eqref{NLKG}. Besides, the instability mechanism of the bound state is revealed through the nonlinear analogous Fermi's Golden Rule \eqref{FGR-SW} as the key resonance coefficient in the dynamical equation of the discrete mode. Thus, this non-negative coefficient reflects the nonlinear resonant interaction between the bound state (eigenfunctions) and the radiation (continuous spectral modes).
	
	The Fermi's Golden Rule condition \eqref{FGR-SW} implies that 
	\begin{equation*}
		3\omega \in \sigma_{c}(B)=[m,+\infty),
	\end{equation*}
	which ensures a strong coupling of the discrete spectrum to the continuous spectrum. For the weak resonance regime, i.e. the discrete spectrum $\omega$ is not close to the continuous spectrum $3\omega < m$, the Fermi's Golden Rule \eqref{FGR-SW} fails, and the method proposed by Soffer and Weinstein \cite{SW1999} cannot be directly applied to derive a similar result. Nevertheless, the authors in \cite{SW1999} conjectured that a similar Fermi's Golden Rule holds, and the solutions decay to zero at certain rates for the general case. Since then, many efforts have been focused on this conjecture, i.e. the instability mechanism and energy transfer rate of metastable states in a weak resonance regime \cite{AS,BC}. We emphasize that the energy transfer rate is crucial to obtain a quantitative description on the lifespan of metastable states.
	
	\subsection{Related Works}	
	For the nonlinear Klein-Gordon equations with zero potential, the global well-posedness and scattering theory for small initial data has been established in several works. For instance, see \cite{Kl} and \cite{Sh} for the three dimensional case, \cite{OTT} and \cite{ST} for the two dimensional case, and \cite{De}, \cite{HN} for the one dimensional case with the modified scattering results. Note that these solutions satisfy the linear dispersive estimates.
	
	When potentials are present $V(x)\not\equiv 0$ but the operator $B$ has no eigenvalues, the same decay results in linear Klein-Gordon equations are true in three space dimensions, see \cite{Yajima}. Several interesting studies focus on quadratic nonlinearities and variable coefficients in one dimensional case, see \cite{GP}, \cite{LLS}. The lower bound of the decay rate has also been proved using an alternative approach in a recent work by An--Soffer \cite{AS} when $B$ has one simple eigenvalue lying close to the continuous spectrum, i.e., the case in \cite{SW1999}. And in the recent interesting work \cite{LP2021}, the authors extended the results of \cite{SW1999} to quadratic nonlinearities and obtained the sharp decay rates. Finally, in the remarkable work \cite{BC}, for initial data in $H^1 \times L^2$ and a general potential which allows for multiply eigenvalues and arbitrary gaps between the discrete and continuous spectra, Bambusi and Cuccagana proposed a new normal form transformation preserving the Hamiltonian structure and proved that small solutions were asymptotically free under a non-degeneracy hypothesis. Moreover, they discovered the instability mechanism of bound states through a nonlinear Fermi's Golden Rule, whose coefficient was never negative and generically strictly positive. We mention that for initial data in $H^1 \times L^2$, decay estimates do not hold due to the conservation of energy. Nevertheless, the authors in \cite{BC} remarked that it was possible to prove appropriate decay rates by restricting initial data to the class in \cite{SW1999}. And in this paper we settle down this problem in our context.
	
	We also mention that the Fermi's Golden Rule was first introduced by Sigal \cite{Sigal} in mathematical context, wherein the author established the instability mechanism of quasi-periodic solutions to nonlinear Schr\"odinger and wave equations. This was then further extended to study the asymptotic stability of bound states of nonlinear Schr\"odinger equations by Tsai--Yau \cite{TY}, Soffer--Weinstein \cite{SW04}, Gang \cite{Gang}, Gang--Sigal \cite{GS}, Gang--Weinstein \cite{GW}; see also the recent advances by Cuccagna--Maeda \cite{CM1}, their survey \cite{CM2} and references therein.

	\subsection{Main Result}
	Our main result is the following theorem:
	
	\begin{thm}\label{main-result}
		Let $(2N-1)\omega<m<(2N+1)\omega$ for some $N\in\bN$ and $C_0$ be a positive constant. Under Assumption \ref{assumption-V}, if the following resonance condition ($\gamma$ is defined in Section \ref{Sec-isolation of the key resonant} )
		\begin{align*}
			\gamma >0
		\end{align*}
		holds, then there exists a sufficiently small constant $\epsilon>0$ such that for any initial data $u_0, u_1$ satisfying 
		$$\|u_0\|_{W^{2,2}\cap W^{2,1}} +\|u_1\|_{W^{1,2}\cap W^{1,1}} \le \epsilon$$ 
		and
		\begin{equation}\label{assumption-data} 
			\|\mathbf{P}_c u_0\|_{W^{2,2}\cap W^{2,1}}+\|\mathbf{P}_c u_1\|_{W^{1,2}\cap W^{1,1}}\le C_0 \left(\|\mathbf{P}_d u_0\|_{L^2}+ \|\mathbf{P}_d u_1\|_{L^2}\right), 
		\end{equation}  
		solution $u(t,x)$ to \eqref{NLKG} with $\lambda\neq 0$  has the following expansion as $t\to \pm\infty$:
		\begin{align*} 
			u(t,x) = R(t) \cos(\omega t+\theta(t))\varphi(x) +\eta(t,x),
		\end{align*}
		where
		\begin{align}
			\frac{\frac{1}{C}R(0)}{(1+4N\lambda^{2N}|R(0)|^{4N}\gamma t)^{\frac{1}{4N}}}\le R(t)\le \frac{CR(0)}{(1+4N \lambda^{2N}|R(0)|^{4N}\gamma t)^{\frac{1}{4N}}} ,\label{R}\\
			\theta(t) = \mathcal{O}(|t|^{1-\frac{1}{2N}}), \quad  \|\eta(t,\cdot)\|_{L^8} = \mathcal{O}(|t|^{-\frac{3}{4N}}).\label{theta}
		\end{align}
		for some positive constant $C>0$. Here, $\mathbf{P}_d$ and $\mathbf{P}_c$ denote the projections onto the discrete and continuous spectral part of $B$ respectively,
	\end{thm}
	
	\begin{rem}
		Evidently, the case where $N = 1$ is exactly the work of Soffer-Weinstein \cite{SW1999}. We point out that the case of $m = (2N + 1)\omega$ is beyond the scope of our work, and it remains unsolved.
	\end{rem}
	
	%	\begin{rem}
		%		The resonance coefficient $\gamma$ in the theorem is more complex than the one in \eqref{FGR-SW} corresponding to the case $N=1$, whose exact expression is given in Section . We emphasize that the non-negative sign of $\gamma$ is crucial to the instability mechanism. However, it does not follow directly by its explicit form, as already observed by Bambusi and Cuccagna in \cite{BC}. In fact, $\gamma$ consists of a combination of terms like $(DF,G)$, where $D$ is a positive operator, but $F,G$ are different functions. This is one of the key difficulties when attempting to obtain the exact order of nonlinearities and is overcome here by a novel use of an ODE type estimate and the conservation of basic energy.
		%	\end{rem}
	
	\begin{rem} We indicate that $R(0)$ can be of order one by choosing a sufficiently small $\lambda$ and a simple scaling argument.
		The assumption on initial data
		$$\|\mathbf{P}_c u_0\|_{W^{2,2}\cap W^{2,1}}+\|\mathbf{P}_c u_1\|_{W^{1,2}\cap W^{1,1}}\lesssim \|\mathbf{P}_d u_0\|_{L^2}+ \|\mathbf{P}_d u_1\|_{L^2},$$
		see \eqref{assumption-data}, is necessary, which leads to resonance-dominated solutions with the decay rates \eqref{R} and \eqref{theta}. Indeed, this assumption should also be imposed in \cite{SW1999} and other related references. As pointed out by Tsai and Yau in \cite{TY}, when
		$$\|\mathbf{P}_c u_0\|_{W^{2,2}\cap W^{2,1}}+\|\mathbf{P}_c u_1\|_{W^{1,2}\cap W^{1,1}}\gg \|\mathbf{P}_d u_0\|_{L^2}+ \|\mathbf{P}_d u_1\|_{L^2},$$ there exist dispersion-dominated solutions with faster decay rates, even in the case of Soffer and Weinstein \cite{SW1999}. 
	\end{rem}
	\begin{rem}
		The decay rate \eqref{R} is consistent with the numerical prediction in \cite{Bi}.
	\end{rem}
	
	\subsection{Outline of Proof}
	Now we explain the ideas of our proof. By setting $u_{\lambda}(t,x)=|\lambda| ^{1/2}u(t,x)$, we may assume $\lambda=\pm 1$. We adapt the framework of Hamiltonian method proposed in \cite{BC} to derive the dynamical equations of discrete and continuous modes. By changing variables, the nonlinear Klein-Gordon equations \eqref{NLKG} can be wirtten as the following Hamilton equations (see Section \ref{Sec-normal form transformation} for details)
	$$
	\dot{\xi}=-\mathrm{i} \frac{\partial H}{\partial \bar{\xi}}, \quad \dot{f}=-\mathrm{i} \nabla_{\bar{f}} H.
	$$
	with the corresponding Hamiltonian
	$$
	\begin{aligned}
		H &=H_{L}+H_{P}, \\
		H_{L} &= \omega\left|\xi\right|^{2}+\langle\bar{f}, B f\rangle, \\
		H_{P} &=\int_{\mathbb{R}^{3}} \left(\sum \frac{\xi+\bar{\xi}}{\sqrt{2 \omega}} \varphi(x)+U(x)\right)^4 d x,
	\end{aligned} $$
	where $\nabla_{\bar{f}} H$ is the gradient with respect to the $L^{2}$ metric, and $U=B^{-\frac{1}{2}}(f+\bar{f}) / \sqrt{2} \equiv P_{c} u$. 
	
	To further simplify the dynamics of discrete mode and decouple the discrete mode from the continuous ones, we proposed a refined Birkhoff normal form transformation compared to the one in \cite{BC}. More precisely, we prove that for any $0\le r\le 2N$ there exists an analytic canonical transformation $\mathcal{T}_{r}$ putting the system in normal form up to order $2r+4$, i.e.
	$$
	H^{(r)}:=H \circ \mathcal{T}_{r}=H_{L}+Z^{(r)}+\mathcal{R}^{(r)}
	$$
	where $Z^{(r)}$ is a polynomial of degree $2r+2$ which is in normal form (see Definition \ref{def:normalform}), and $\mathcal{R}^{(r)}$ consists of error terms of higher order, see Theorem \ref{thm:nft}. The normal form transformation is canonical and preserves the Hamiltonian nature of the system. Our key observation is that the order of normal form is actually increased by two in each step, which is consistent with the theoretic prediction by the non-Hamiltonian method of Soffer and Wesinsten \cite{SW1999}. Besides, we explore the explict forms of theses coefficients appeared in error terms, whose structure will be crutial in  error estimates.

	%	To get decay estimates of the solution, it is necessary to obtain an ODE type estimates for the discrete modes. When $N=1$, Soffer and Weinstein \cite{SW1999} obtained the desired ODE by giving an asymptotic expansion of the continuous modes $\eta$, but when $N$ gets large, it is not likely to do so directly (because the larger $N$ gets, the higher order expansions of $\eta$ we need). To overcome this difficulty, we apply the Birkhoff normal form transformation method as in \cite{BC} to normalize the original equation.

	After applying the normal form transformation, we are able to extract the main terms of the continuous modes $f$ and obtain the dynamics of the discret mode as follows: 
	\begin{equation*}
		\frac{d}{dt}|\xi|^2=-2\gamma |\xi|^{4N+2}+2Re(\bar{\xi}\mathcal{R}_{\xi}),
	\end{equation*}
	where
	\begin{align*}
		\gamma :=& (2N+1)\langle \Phi_{0,2N+1},  \delta(B-(2N+1)\omega)\bar{\Phi}_{0,2N+1}\rangle \ge 0, \quad \Phi_{0,2N+1} \in \mathcal{S}\left(\mathbb{R}^{3}, \mathbb{C}\right),
	\end{align*}
	is the coefficient of nonlinear Fermi's Golden Rule, and $\mathcal{R}_{\xi}$ is an error term, see \eqref{eq:Rxi}. Thus, under the Fermi's Golden Rule condition $\gamma>0$, we could expect that the decay rate of the discrete mode $\xi(t)$ is $O(|t|^{-\frac{1}{4N}})$, provided that an appropriate estimate for the error term $\mathcal{R}_{\xi}$ is present.
	
	To close our argument, the estimate of error term $\mathcal{R}_{\xi}$ remains to be performed, which by \eqref{eq:Rxi} requires deriving the decay estimates of the continuous mode $f$. Once this is achieved, a bootstrap argument
	would imply Theorem \ref{main-result}. Since we are pursuing decay estimates in time, the space-time dispersive estimates from \cite{BC} are not applicable here. Besides, we shall see that the classical $L^p$ estimates also fail due to the slow decay rate of $\xi$ and $f$, especially for the weak resonance regime $N> 1$. To illustrate this, consider
	$$\hat{\mathcal{R}}_f :=  -\mathrm{i} \int_{0}^{t}e^{-\mathrm{i}B (t-s)}\partial_{\bar{f}}\mathcal{R}ds,$$
	which appears as an error term in the expression of the continuous mode $f$, see \eqref{f formula}. Here $\mathcal{R}$ is the error term coming from the normal form transformation 	
	$\mathcal{R}=\sum_{d=1}^{5}\mathcal{R}_{d}.$
	Denote
	$$\hat{\mathcal{R}}^2_f :=  -\mathrm{i} \int_{0}^{t}e^{-\mathrm{i}B (t-s)}\partial_{\bar{f}}\mathcal{R}_2ds.$$	
	%Our main difficulty is to estimate the error term $\mathcal{R}_{\xi}.$ If the error term $\mathcal{R}_{\xi}$ vanishes, \eqref{eq:xi3} implies  $|\xi|\approx (1+t)^{-\frac{1}{4N}}, |\xi'|\approx (1+t)^{-\frac{4N+1}{4N}} $, hence to treat $\mathcal{R}_{\xi}$ as small perturbation, it requires $\mathcal{R}_{\xi}\approx (1+t)^{-\frac{4N+1}{4N}-\delta}$, see Theorem \ref{ode}. To prove $\mathcal{R}_{\xi}$ has such decay rate, it is necessary to prove $f$ decays like $(1+t)^{-\frac{2N+1}{4N}}$ in some norm. Write $f=f_1+f_2+f_3$, where 
	%	\begin{align*}
		%		f_1&:=e^{-\mathrm{i}B t}f(0)ds,\\
		%		f_2&:=-\mathrm{i}\int_{0}^{t}e^{-\mathrm{i}B (t-s)}\bar{G}ds,\\
		%		f_3&:=-\mathrm{i}\int_{0}^{t}\partial_{\bar{f}}\mathcal{R}ds,
		%	\end{align*}
	%	the estimate of $f_3$ is the most difficult. 	
	%To proceed, we further expand  $$f_3=\sum_{d=1}^{5}f_{3d}:=-\mathrm{i}\sum_{d=1}^{5}\int_{0}^{t}\partial_{\bar{f}}\mathcal{R}_d ds.$$ We take $d=2$ for example. 
	When $N=1$ as in \cite{SW1999}, using linear dispersive estimates, and by virtue of the structure of $\mathcal{R}_d$ in Theorem \ref{thm:nft}, one has	
	\begin{align*}
		\|B^{-1/2}\hat{\mathcal{R}}^2_f\|_{L^{8}} & \lesssim \int_{0}^{t} \min \left\{|t-s|^{-\frac{9}{8}},|t-s|^{-\frac{3}{8}}\right\}|s|^{-\frac{1}{2}}[\xi]_{\frac{1}{4}}^{2}\left(\|B^{-1/2}f\|_{L^{8}}+\|B^{1/2}f\|_{L^{4}}\right)ds, \\
		\|B^{1/2} \hat{\mathcal{R}}^2_f\|_{L^{4}} & \lesssim \int_{0}^{t}|t-s|^{-\frac{1}{2}}|s|^{-\frac{1}{2}}[\xi]_{\frac{1}{4}}^{2}\left(\|B^{-1/2}f\|_{L^{8}}+\|B^{1/2}f\|_{L^{4}}\right)ds,
	\end{align*}
	which yield
	\begin{align*}
		\|B^{-1/2}\hat{\mathcal{R}}^2_f\|_{L^{8}} & \lesssim |t|^{-1+\delta_{0}}[\xi]_{\frac{1}{4}}^{2}\left([B^{-1/2}f]_{8, \frac{3}{4}}+[B^{1/2} f]_{4,1 / 2-\delta_{0}}\right), \\
		\|B^{1/2} \hat{\mathcal{R}}^2_f\|_{L^{4}} & \lesssim |t|^{-\frac{1}{2}+\delta_{0}}[\xi]_{\frac{1}{4}}^{2}\left([B^{-1/2}f]_{8, \frac{3}{4}}+[B^{1/2} f]_{4,1 / 2-\delta_{0}}\right),
	\end{align*}
	where $[\xi]_{\alpha}=\sup _{0 \leq t \leq T}\langle t\rangle^{\alpha}|\xi(t)|,[g]_{p, \alpha}=\sup _{0 \leq t \leq T}\langle t\rangle^{\alpha}\|g\|_{L^{p}}$. Together with decay estimate of $\xi(t)=O(|t|^{-\frac 14})$, these estimates imply the desired error estimate of $\mathcal{R}_{\xi}$ for small data, which  is similar to the approach used in \cite{SW1999}. We emphasize that in the aforementioned estimates, the decay rate $O(|t|^{-\frac 14})$ of $\xi(t)$ is necessary and cannot be weakened.
	However, for the weak resonance regime $N>1$, the expected decay rate $\xi(t)=O(|t|^{-\frac{1}{4N}})$, which is much slower than required.  For instance, the linear dispersive estimates imply
	$$
	\|B^{1/2} \hat{\mathcal{R}}^2_f\|_{L^{4}} \lesssim \int_{0}^{t}|t-s|^{-\frac{1}{2}}|s|^{-\frac{1}{2N}}[\xi]_{\frac{1}{2N}}^{2}\left(\|B^{-1/2}f\|_{L^{8}}+\|B^{1/2}f\|_{L^{4}}\right),
	$$
	where
	$$
	\frac{1}{2}+\frac{1}{2N}<1.
	$$
	Hence, the decay estimates of $B^{-1/2}f, B^{1/2}f$ cannot be closed this way. 
	
	In this paper, we overcome this difficulty by applying the weighted $L^p$ estimates, see Lemma \ref{Weighted-Lp-Lp'}, which acquires time decay at the cost of the spatial decay. More precisely, instead of estimating $\|B^{1/2} f\|_{L^4}$, we consider the weighted $L^4$ norm $ \|\langle x\rangle^{-\sigma}B^{1/2} f\|_{L^4}$, then
	\begin{align*} 
		\|\langle x\rangle^{-\sigma}B^{1/2} f_{32}\|_{L^4}
		\lesssim & \int_0^t \min\{|t-s|^{-\frac 54}, |t-s|^{-\frac 12}\} |s|^{-\frac{1}{2N}}[\xi]^2_{\frac{1}{4N}} \left\| \langle x\rangle^{\sigma}\partial_{\bar{f}}\mathcal{R}_2\right\|_{W^{1,\frac 43}}\\
		\lesssim & \langle |\xi_0|^{4N}t \rangle^{-\frac{2N+2}{4N}}[\xi]^{2}_{\frac{1}{4N}}\big([\langle x \rangle^{-\sigma}B^{1/2}f]_{4,\frac{1}{2}}+[B^{-1/2}f]_{8,\frac{2N+1}{4N}}\big) .
	\end{align*}	
	Such a technique is based on the spatially localized property of $\partial_{\bar{f}}\mathcal{R}_2$ to absorb the weight $\langle x\rangle^{\sigma}$ on the RHS.

	\subsection{Structure of the Paper}
	The remaining part of this paper is organized as follows. In Section \ref{Sec-Preliminary}, we introduce some useful dispersive estimates and weighted inequalities for linear equations with potential, and we present the global existence theory and energy conservation of the nonlinear Klein-Gordon equation. In Section \ref{Sec-normal form transformation}, we begin our proof by performing Birkhoff normal form transformation. Then we isolate the key resonant terms in the dynamical equation of the discrete mode in Section \ref{Sec-isolation of the key resonant}. In Section \ref{Sec-aymptotic behavior}, we derive the asymptotic behavior of discrete and continuous modes. In Section \ref{Sec-error estimates}, the error terms are carefully estimated. In Section \ref{Sec-proof of main result}, we prove our main result using apriori estimates and bootstrap arguments.
	
	\subsection{Notations}
	Throughout our paper, we use $C$ to denote an absolute positive constant that may vary from line to line. We write $ A\lesssim B$ to mean that $A\le CB$ for some absolute constant $C>0$. We will use $A\approx B$ in a similar standard way.

	\section{Preliminary}\label{Sec-Preliminary} 
	In this section, we provide some useful lemmas on the linear analysis for the Klein-Gordon equation with potential and the global well-posedness theory of the nonlinear Klein-Gordon equation \eqref{NLKG}.
	
	\subsection{Linear Dispersive Estimates}
	Consider the Cauchy problem for three dimensional linear Klein-Gordon equation with a potential
	\begin{equation}\label{LKW}
		\begin{cases}
			\partial_t^2 u -\Delta u + m^2 u + V(x)u = 0, & \quad t>0, x\in \bR^3,\\
			u(0,x)=u_0,\quad \partial_{t} u(0,x)=u_1.
		\end{cases}
	\end{equation}
	Denote $B^2= -\Delta  + m^2  + V(x)$, then equation \eqref{LKW} can be solved as
	\begin{align*}
		u(t,x)= \cos{Bt}\ u_0 + \frac{\sin Bt}{B}\  u_1.
	\end{align*}

	For $V(x)=0$, i.e. free Klein-Gordon case, the standard $L^p$ dispersive estimates follow from an oscillatory integration method and the conservation of the $L^2$ norm. More precisely, the $L^p$ norm of the solution to $u(t,x)$ satisfies the dispersive decay estimate $\|u(t, \cdot)\|_{L^p} \leq C|t|^{-3(\frac12-\frac{1}{p})}$. %Here, and in what follows, $C$ is used to denote a generic positive constant whose meaning may change from line to line. 
	
	For $V(x)\neq 0$, if $V(x)$ satisfies some suitable decay and regularity conditions, then the same decay rate of $u$ can be obtained by the $W^{k,p}$-boundedness of the wave operator after being projected on the continuous spectrum of $B$. For instance, see \cite{JSS},\cite{RS},\cite{Yajima}. 
	\begin{lem}[$L^p$ dispersive estimates]\label{Lp-Lp'}
		Assume that $V(x)$ is a real-valued function and satisfies (V1),(V2). Let $1<p\le 2$, $\frac 1p+\frac{1}{p'}=1$, $0\le \theta \le 1$, $l=0,1$,  and $s= (4+\theta)(\frac 12-\frac{1}{p'})$. Then
		\begin{equation*}
			\|\mathrm{e}^{\mathrm{i}Bt}B^{-l}\mathbf{P}_{\mathrm{c}}\psi\|_{l,p'}\lesssim |t|^{-(2+ \theta)(\frac12-\frac{1}{p'})}\|\psi\|_{s,p}, \quad |t|\ge 1,
		\end{equation*}
		and
		\begin{equation*}
			\|\mathrm{e}^{\mathrm{i}Bt}B^{-l}\mathbf{P}_{\mathrm{c}}\psi\|_{l,p'}\lesssim |t|^{-(2- \theta)(\frac12-\frac{1}{p'})}\|\psi\|_{s,p}, \quad 0<|t|\le 1.
		\end{equation*}
	\end{lem}
	
	The following weighted decay estimate of the Klein-Gordon equation was first established by Jensen and Kato \cite{JS} for the Schr\"odinger equation and then extended to the Klein-Gordon equation by Komech and Kopylova \cite{KK}. These estimates are used to reveal the non-resonance structure of nonlinearities in the equation satisfied by the continuous spectrum part.
	\begin{lem}[Weighted $L^2$ estimate]\label{Weighted-L2} Assume $V(x)$ satisfies the hypothesis of Lemma \ref{Lp-Lp'}. Then, for $\sigma>\frac 52$, $l=0,1$, we have
		\begin{equation*} 
			\left\|\langle x\rangle^{-\sigma} \mathrm{e}^{\mathrm{i} B t}B^{-l} \mathbf{P}_{\mathrm{c}}\langle x\rangle^{-\sigma} \psi\right\|_{l,2} \lesssim \langle t\rangle ^{-\frac32}\|\psi\|_{2}.
		\end{equation*}		
	\end{lem}
	The following weighted $L^p$ estimates, which are consequences of the interpolation between the standard $L^\infty-L^{1}$ estimates and the weighted $L^2$ estimate, are also necessary:
	\begin{lem}[Weighted $L^p$ estimate]\label{Weighted-Lp-Lp'}
		Assume that $V(x)$ satisfies the hypothesis of Lemma \ref{Lp-Lp'}. Let $1<p\le 2$, $\frac 1p+\frac{1}{p'}=1$, $0\le \theta \le 1$, $l=0,1$, and $s= (4+\theta)(\frac 12-\frac{1}{p'})$. Then, for $\sigma>\frac 52$,
		\begin{equation*} 
			\left\|\langle x\rangle^{-\sigma} \mathrm{e}^{\mathrm{i} B t}B^{-l} \mathbf{P}_{\mathrm{c}}\langle x\rangle^{-\sigma} \psi\right\|_{l,p'} \lesssim |t|^{-(2+ \theta)(\frac12-\frac{1}{p'})-\frac{3}{p'}}\|\psi\|_{s,p}, \quad  |t|\ge 1
		\end{equation*}
		and
		\begin{equation*}  
			\left\|\langle x\rangle^{-\sigma} \mathrm{e}^{\mathrm{i} B t}B^{-l} \mathbf{P}_{\mathrm{c}}\langle x\rangle^{-\sigma} \psi\right\|_{l,p'} \lesssim |t|^{-(2- \theta)(\frac12-\frac{1}{p'})}\|\psi\|_{s,p}, \quad 0<|t|\le 1
		\end{equation*}
		hold.
	\end{lem}

	%In particular, we use Lemma \ref{Weighted-Lp-Lp'} for $p'=4, \theta=0, l=1$, which is
	%	\begin{cor}\label{cor-weighted-L4} Under the same hypothesis of Lemma \ref{Lp-Lp'}, for $\sigma>\frac 52$, 
		%		\begin{equation*} 
			%			\left\|\langle x\rangle^{-\sigma} \mathrm{e}^{\mathrm{i} B t}B^{-1} \mathbf{P}_{\mathrm{c}}\langle x\rangle^{-\sigma} \psi\right\|_{1,4} \lesssim |t|^{-\frac12}\langle t\rangle ^{-\frac{3}{4}}\|\psi\|_{1,4/3}	
			%		\end{equation*}
		%		holds.
		%	\end{cor}
	
	The use of the weighted $L^p$ estimates enables us to obtain a better time decay of the continuous spectrum part of $u$ from its spatial decay, which aids dealing with the asymptotic behavior of solutions to the nonlinear Klein-Gordon equations. In particular, they help us to overcome the loss of the derivative of $\eta_3$. For more details, see  Section 7.

	\subsection{Singular Resolvents and Time Decay}
	The following local decay estimates for singular resolvents $\mathrm{e}^{\mathrm{i} B t}(B-\Lambda+\mathrm{i} 0)^{-l}$,  which was proved in \cite{SW1999}, are also significant. Here, $\Lambda$ is a point in the interior of the continuous spectrum of $B(\Lambda>m)$.
	\begin{lem}[Decay estimates for singular resolvents]\label{lem-singular-resolvents}
		Assume that $V(x)$ is a real-valued function and satisfies (V1)-(V3). Let $\sigma>16/5$. Then for any point $\Lambda>m$ in the continuous spectrum of $B$, we have for $l=1,2:$ 
		\begin{align*}
			&\left\|\langle x\rangle^{-\sigma} \mathrm{e}^{\mathrm{i} B t}(B-\Lambda+\mathrm{i} 0)^{-l} \mathbf{P}_{\mathrm{c}}\langle x\rangle^{-\sigma} \psi\right\|_{2} \lesssim \langle t\rangle^{-\frac{6}{5}}\|\psi\|_{1,2}, \quad t>0,\\	
			&\left\|\langle x\rangle^{-\sigma} \mathrm{e}^{\mathrm{i} B t}(B-\Lambda-\mathrm{i} 0)^{-l} \mathbf{P}_{\mathrm{c}}\langle x\rangle^{-\sigma} \psi\right\|_{2} \lesssim \langle t\rangle^{-\frac{6}{5}}\|\psi\|_{1,2}, \quad t<0.
		\end{align*}
	\end{lem}

	\subsection{Global Well-Posedness and Energy Conservation}
	The global well-posedness of \eqref{NLKG} with small initial data is well-known. 
	\begin{thm}
		Assume $V \in L^{p}$ with $p>3 / 2$. Then, there exists $\varepsilon_{0}>0$ and $C>0$, such that for any $\left\|\left(u_{0}, u_{1}\right)\right\|_{H^{1} \times L^{2}} \leq \epsilon<\varepsilon_{0}$, equation \eqref{NLKG} admits exactly one solution
		$u \in C^{0}\left(\mathbb{R}; H^{1}\right) \cap C^{1}\left(\mathbb{R}; L^{2}\right)$
		such that $(u(0), \partial_{t}u(0))=\left(u_{0}, u_{1}\right)$. Furthermore, the map $\left(u_{0}, u_{1}\right) \mapsto (u(t), \partial_{t}u(t))$ is continuous from the ball $\left\|\left(u_{0}, u_{1}\right)\right\|_{H^{1} \times L^{2}}<\varepsilon_{0}$ to $C^{0}\left(I; H^{1}\right) \times C^{0}\left(I; L^{2}\right)$ for any bounded interval $I$. Moreover, the energy 
		\begin{align*}
			\mathcal{E}[u,\partial_{t} u]\equiv \frac12\int (\partial_t u)^2 + |\nabla u|^2 + m^2 u^2 + V(x)u^2 dx - \frac{\lambda}{4}  \int u^4 dx.
		\end{align*}
		is conserved and
		$$
		\|(u(t), v(t))\|_{H^{1} \times L^{2}} \leq C\left\|\left(u_{0}, v_{0}\right)\right\|_{H^{1} \times L^{2}} .
		$$
	\end{thm}
	We refer to \cite{CH} for details.

	\section{Normal Form Transformation}\label{Sec-normal form transformation}
	In this section, we present a new normal form transformation which is a refined version of Theorem 4.9 in \cite{BC}. The main differences are as follows: (i) we find that the order of normal form actually increase by two in each step, which is consistent with the result in \cite{SW1999}; (ii) we give explicit forms of these coefficients appeared in error terms, whose structure will be crucial in the subsequent error estimates.
	\subsection{Hamiltonian Structure}	
	Now we consider 3D nonlinear Klein Gordon equation (NLKG)
	\begin{align}
		u_{t t}-\Delta u+V u+m^{2} u =  u^3, \quad(t, x) \in \mathbb{R} \times \mathbb{R}^{3}, 
	\end{align}
	which is an Hamiltonian perturbation of linear Klein-Gordon equation with potential. More precisely, in $H^{1}\left(\mathbb{R}^{3}, \mathbb{R}\right) \times L^{2}\left(\mathbb{R}^{3}, \mathbb{R}\right)$ endowed with the standard symplectic form, namely
	$$
	\Omega\left(\left(u_{1}, v_{1}\right) ;\left(u_{2}, v_{2}\right)\right):=\left\langle u_{1}, v_{2}\right\rangle_{L^{2}}-\left\langle u_{2}, v_{1}\right\rangle_{L^{2}}
	$$
	we consider the Hamiltonian
	$$
	\begin{aligned}
		H &=H_{L}+H_{P}, \\
		H_{L} &:=\int_{\mathbb{R}^{3}} \frac{1}{2}\left(v^{2}+|\nabla u|^{2}+V u^{2}+m^{2} u^{2}\right) d x, \\
		H_{P} &:= \int_{\mathbb{R}^{3}} -\frac{1}{4} u^4  d x.
	\end{aligned}
	$$ 
	The corresponding Hamilton equations are $\dot{v}=-\nabla_{u} H, \dot{u}=\nabla_{v} H$, where $\nabla_{u} H$ is the gradient with respect to the $L^{2}$ metric, explicitly defined by
	$$
	\left\langle\nabla_{u} H(u), h\right\rangle=d_{u} H(u) h, \quad \forall h \in H^{1}
	$$
	and $d_{u} H(u)$ is the Frechét derivative of $H$ with respect to $u$. It is easy to see that the Hamilton equations are explicitly given by
	\begin{align*}
		\left(\dot{v}=\Delta u-V u-m^{2} u +  u^3, \dot{u}=v\right) \Longleftrightarrow \ddot{u}=\Delta u-V u-m^{2} u +  u^3.
	\end{align*}
	%	Under suitable assumptions on the potential $V(x)$, the Schr\"odinger operator $-\Delta+ V + m^2$ has finite eigenvalues $0<\omega_1^2\le \omega_2^2\le \cdots \le \omega_n^2 < m^2,$ with the corresponding eigenfunctions $\varphi_j(x)$, i.e. $$-\Delta \varphi_j+ V \varphi_j + m^2 \varphi_j = \omega_j^2 \varphi_j.$$
	Write $$u= q \varphi + P_c u, \quad v= p  \varphi + P_c v,$$
	with a slightly abuse of notations, from now on we denote
	$$
	B:=P_{c}\left(-\Delta+V+m^{2}\right)^{1 / 2} P_{c},
	$$
	and define the complex variables
	\begin{equation}
		\xi:=\frac{q \sqrt{\omega}+\mathrm{i} \frac{p}{\sqrt{\omega}}}{\sqrt{2}}, \quad f:=\frac{B^{1 / 2} P_{c} u+\mathrm{i} B^{-1 / 2} P_{c} v}{\sqrt{2}}.
	\end{equation}
	Then, in terms of these variables the symplectic form has the form
	$$
	\begin{array}{r}
		\Omega\left(\left(\xi^{(1)}, f^{(1)}\right) ;\left(\xi^{(2)}, f^{(2)}\right)\right)=2 \operatorname{Re}\left[\mathrm{i}\left( \xi^{(1)} \bar{\xi}^{(2)}+\left\langle f^{(1)}, \bar{f}^{(2)}\right\rangle\right)\right] \\
		=-\mathrm{i} \sum_{j}\left(\bar{\xi}^{(1)} \xi^{(2)}-\xi^{(1)} \bar{\xi}^{(2)}\right)-\mathrm{i}\left(\left\langle f^{(2)}, \bar{f}^{(1)}\right\rangle-\left\langle f^{(1)}, \bar{f}^{(2)}\right\rangle\right)
	\end{array}
	$$
	and the Hamilton equations take the form
	$$
	\dot{\xi}=-\mathrm{i} \frac{\partial H}{\partial \bar{\xi}}, \quad \dot{f}=-\mathrm{i} \nabla_{\bar{f}} H.
	$$
	where
	$$
	\begin{gathered}
		H_{L}= \omega\left|\xi\right|^{2}+\langle\bar{f}, B f\rangle, \\
		H_{P}(\xi, f)=\int_{\mathbb{R}^{3}} \left(\sum \frac{\xi+\bar{\xi}}{\sqrt{2 \omega}} \varphi(x)+U(x)\right)^4 d x
	\end{gathered}
	$$
	with $U=B^{-\frac{1}{2}}(f+\bar{f}) / \sqrt{2} \equiv P_{c} u$.
	The Hamiltonian vector field $X_{H}$ of a function is given by
	$$
	X_{H}(\xi, \bar{\xi}, f, \bar{f})=\left(-\mathrm{i} \frac{\partial H}{\partial \bar{\xi}}, \mathrm{i} \frac{\partial H}{\partial \xi},-\mathrm{i} \nabla_{\bar{f}} H, \mathrm{i} \nabla_{f} H\right).
	$$
	The associate Poisson bracket is given by
	$$
	\begin{aligned}
		\{H, K\}:=& \mathrm{i} \left(\frac{\partial H}{\partial \xi} \frac{\partial K}{\partial \bar{\xi}}-\frac{\partial H}{\partial \bar{\xi}} \frac{\partial K}{\partial \xi}\right) \\
		&+\mathrm{i}\left\langle\nabla_{f} H, \nabla_{\bar{f}} K\right\rangle-\mathrm{i}\left\langle\nabla_{\bar{f}} H, \nabla_{f} K\right\rangle .
	\end{aligned}
	$$
	Denote $z=(\xi,f), \mathbf{f}=(f,\bar{f}),$ and $\mathcal{P}^{k, s}=\mathbb{C} \times P_{c} H^{k, s}\left(\mathbb{R}^{3}, \mathbb{C}\right)$, where $$H^{k, s}\left(\mathbb{R}^{3}, \mathbb{C}\right)=\left\{f: \mathbb{R}^{3} \rightarrow \mathbb{C} \text { s.t. }\|f\|_{H^{s, k}}:=\left\|\langle x\rangle^{s}(-\Delta+1)^{k / 2} f\right\|_{L^{2}}<\infty\right\}.$$
	
	\subsection{Lie Transform}	
	%We will iteratively eliminate from the Hamiltonian monomials, simplifying the part linear in $f$ and $\bar{f}$ and the part independent of such variables. We will use canonical transformations generated by Lie transform, namely the time 1 flow of a suitable auxiliary Hamiltonian function. 
	Consider a function $\chi$ of the form
	\begin{align}\label{eq:chi}
		\chi(z) \equiv \chi(\xi, f)=\chi_0(\xi, \bar{\xi})+\sum_{|\mu|+|\nu|=M_0+1} \xi^\mu \overline{\xi^\nu} \int_{\mathbb{R}^3} \mathbf{\Phi}_{\mu \nu} \cdot \mathbf{f} d x
	\end{align}
	where $\mathbf{\Phi}_{\mu \nu} \cdot \mathbf{f}:=\Phi_{\mu \nu} f+\Psi_{\mu\nu} \bar{f}$ with $\Phi_{\mu \nu}, \Psi_{\mu \nu} \in \mathcal{S}\left(\mathbb{R}^3, \mathbb{C}\right)$ and where $\chi_0$ is a homogeneous polynomial of degree $M_0+2$. The Hamiltonian vector field satisfies $X_\chi \in C^{\infty}\left(\mathcal{P}^{-\kappa,-s}, \mathcal{P}^{k, \tau}\right)$ for any $k, \kappa, s, \tau \geq 0$. Moreover we have
	\begin{align}\label{eq:X-chi}
		\left\|X_\chi(z)\right\|_{\mathcal{P}^{k, \tau}} \leq C_{k, s, \kappa, \tau}\|z\|_{\mathcal{P}^{-\kappa,-s}}^{M_0+1} .
	\end{align}
	Since $X_\chi$ is a smooth polynomial it is also analytic. Denote by $\phi^t$ the flow generated by $X_\chi$. For fixed $\kappa, s,$ by \eqref{eq:X-chi} $\phi^t$ is well defined up to any fixed time $\bar{t}$, in a sufficiently small neighborhood $\mathcal{U}^{-\kappa,-s}\subset \mathcal{P}^{-\kappa,-s}$ of the origin. Set $\phi:=\left.\phi^1 \equiv \phi^t\right|_{t=1}$. The canonical transformation $\phi$ will be called the Lie transform generated by $\chi$.
	
	\begin{lem}\label{lem:z-expansion} Given a functional $\chi$ of the form \eqref{eq:chi}. Assume $\Phi_{\mu \nu}, \Psi_{\mu\nu} \in$ $\mathcal{S}\left(\mathbb{R}^3, \mathbb{C}\right)$ for all $\mu$ and $\nu$. Let $\phi$ be its Lie transform. Denote $z^{\prime}=\phi(z), z \equiv(\xi, f)$ and $z^{\prime} \equiv\left(\xi^{\prime}, f^{\prime}\right)$. Then, there exists a sufficiently small neighborhood $\mathcal{U}^{-\kappa,-s}\subset \mathcal{P}^{-\kappa,-s}$ of the origin, such that the following expansions hold:
		\begin{align}
			&\xi' = \xi + \sum_{k=1}^{\infty}\sum_{i=0}^{k}\sum_{\mu+\nu=kM_0+1-i}a_{i\mu \nu} \xi^{\mu}\bar{\xi}^{\nu}\prod_{j=1}^{i}\int \mathbf{\Phi}_{\mu\nu}^{ij}\cdot \mathbf{f}dx,\\
			&f'  = f+ \sum_{k=1}^{\infty}\sum_{i=0}^{k-1}\sum_{\mu+\nu=kM_0+1-i} b_{i\mu \nu}\xi^{\mu}\bar{\xi}^{\nu}\prod_{j=1}^{i}\int \mathbf{\Lambda}_{\mu\nu}^{ij}\cdot \mathbf{f}dx \Psi_{\mu \nu}^{i}.
		\end{align}
		where $a_{i\mu \nu}$ and $b_{i\mu \nu}$ are constants and $\mathbf{\Phi}_{\mu\nu}^{ij}, \mathbf{\Lambda}_{\mu\nu}^{ij}, \Psi_{\mu \nu}^{i}\in \mathcal{S}\left(\mathbb{R}^{3}, \mathbb{C}\right).$
	\end{lem}

	\begin{proof}
		Let $z(t)=(\xi(t),f(t))=\phi^t(z)$, then $z(0)= z, z(1)= z'$. By virtue of the Lie transform and the analyticity of the Hamiltonian vector field $X_{\chi}$, we have
		\begin{equation}
			\xi' = \xi +\sum_{k= 1}^{\infty} \frac{1}{k!}\underbrace{\{\chi_r,\dots,\{\chi_r}_{k \text{ times}},\xi\}\},
			%+ \int_0^1\frac{(1-s)^n}{n!} \underbrace{\{\chi_r,\dots,\{\chi_r}_{n \text{ times}},\xi\}\}\circ \phi_r^s ds	
		\end{equation}
		and it is easy to show by induction that 
		\begin{align*}
			\underbrace{\{\chi_r,\dots\{\chi_r}_{k \text{ times}},\xi\}\} =& \sum_{i=0}^{k}\sum_{\mu+\nu=kM_0+1-i}a_{i\mu \nu} \xi^{\mu}\bar{\xi}^{\nu}\prod_{j=1}^{i}\int \mathbf{\Phi}_{\mu\nu}^{ij}\cdot \mathbf{f}dx. 
		\end{align*}
		Similarly, since 
		\begin{equation}
			\frac{df}{dt}=-i\sum_{\mu+\nu=M_0+1}\xi^{\mu}\bar{\xi}^{\nu}\Psi_{\mu\nu},
		\end{equation}
		we have 
		\begin{equation}
			f' = f -i	\sum_{\mu+\nu=M_0+1}\int_0^1 \xi^{\mu}(t)\bar{\xi}^{\nu}(t)dt \Psi_{\mu \nu}.
		\end{equation}
		By virtue of the Lie transform, we have 
		\begin{equation}	\xi^{\mu}(t)\bar{\xi}^{\nu}(t) =  \sum_{k=0}^{\infty} \frac{t^k}{k!}\underbrace{\{\chi_r,\dots\{\chi_r}_{k \text{ times}},\xi^{\mu}\bar{\xi}^{\nu}\}\},
		\end{equation}
		hence
		\begin{equation}
			f' = f -i	\sum_{\mu+\nu=M_0 +1}\sum_{k=0}^{\infty} \frac{1}{(k+1)!}\underbrace{\{\chi_r,\dots\{\chi_r}_{k \text{ times}},\xi^{\mu}\bar{\xi}^{\nu}\}\} \Psi_{\mu\nu}.
		\end{equation}
		The rest follows similarly.
	\end{proof}	
	
	\subsection{Normal Form Transformation}
	\begin{defn}\label{def:normalform}
		A polynomial $Z$ is in normal form if we have
		$$
		Z=Z_{0}+Z_{1}
		$$
		where $Z_{0}$ is a linear combination of monomials $|\xi|^{2\mu}$, and $Z_{1}$ is a linear combination of monomials of the form
		$$
		\xi^{\mu} \overline{\xi^{\nu}} \int \Phi(x) f(x) d x, \quad \xi^{\mu^{\prime}} \overline{\xi^{\nu^{\prime}}} \int \Phi(x) \bar{f}(x) d x
		$$
		with indexes satisfying
		$$
		\omega (\mu-\nu)<-m, \quad \omega \left(\mu^{\prime}-\nu^{\prime}\right)>m
		$$
		and $\Phi \in \mathcal{S}\left(\mathbb{R}^{3}, \mathbb{C}\right).$
	\end{defn}

	\begin{thm}\label{thm:nft}
		For any $n>0,s>0$ and any integer $r$ with $0 \leq r \leq 2 N$, there exist open neighborhoods of the origin $\mathcal{U}_{r, n, s} \subset \mathcal{P}^{1 / 2,0}$, $\mathcal{U}_{r}^{-n,-s} \subset \mathcal{P}^{-n,-s}$, and an analytic canonical transformation $\mathcal{T}_{r}: \mathcal{U}_{r, n, s} \rightarrow \mathcal{P}^{1 / 2,0}$, such that $\mathcal{T}_{r}$ puts the system in normal form up to order $2r+4$. More precisely, we have
		$$
		H^{(r)}:=H \circ \mathcal{T}_{r}=H_{L}+Z^{(r)}+\mathcal{R}^{(r)}
		$$
		where:
		(i) $Z^{(r)}$ is a polynomial of degree $2r+2$ which is in normal form,\\
		%furthermore, when we expand $$Z_{1}^{(r)}(\xi, f)=\sum_{\substack{\mu, \nu,\\ \omega\cdot(\mu-\nu)<-m}} \xi^{\mu} \bar{\xi}^{\nu} \int_{\mathbb{R}^{3}} \Phi_{\mu \nu} f d x+\sum_{\substack{\mu, \nu,\\ \omega\cdot(\nu-\mu)>m}} \bar{\xi}^{\mu} \xi^{\nu} \int_{\mathbb{R}^{3}} \bar{\Phi}_{\mu \nu} \bar{f} d x,$$we have, $\varphi^{\mu}=\prod_{j} \varphi_{j}^{\mu_{j}}$, $\omega^{\mu}=\prod_{j} \omega_{j}^{\mu_{j}}$, and $$\mathcal{S}\left(\mathbb{R}^{3}, \mathbb{C}\right) \ni \Phi_{\mu 0}=\frac{2^{-\frac{|\mu|}{2}}}{\mu !} \beta_{|\mu|+1} \frac{B^{-\frac{1}{2}}\left(\varphi^{\mu}\right)(x)}{\sqrt{\omega^{\mu}}}+\tilde{\Phi}_{\mu 0}$$with $\tilde{\Phi}_{\mu 0}=\tilde{\Phi}_{\mu 0}\left(m, \beta_{4}, \ldots, \beta_{|\mu|}\right)$ piecewise smooth in $\left(m, \beta_{4}, \ldots, \beta_{|\mu|}\right)$, with values in $\mathcal{S}\left(\mathbb{R}^{3}, \mathbb{C}\right)$; the functions $\Phi_{\mu \nu}(x)\in \mathcal{S}\left(\mathbb{R}^{3}, \mathbb{C}^{2}\right)$;
		(ii) $I-\mathcal{T}_{r}$ extends into an analytic map from $\mathcal{U}_{r}^{-n,-s}$ to $\mathcal{P}^{n, s}$ and
		$$
		\left\|z-\mathcal{T}_{r}(z)\right\|_{\mathcal{P}^{n, s}} \lesssim \|z\|_{\mathcal{P}^{-n,-s}}^{3}.
		$$
		(iii) we have $\mathcal{R}^{(r)}=\sum_{d=0}^{5} \mathcal{R}_{d}^{(r)}$ with the following properties: \\
		(iii.0) we have
		$$
		\mathcal{R}_{0}^{(r)}=\sum_{\mu+\nu=2r+4} a_{\mu \nu}^{(r)}\left(\xi\right) \xi^{\mu} \bar{\xi}^{\nu}  
		$$
		where $a_{\mu \nu}^{(r)}\in C^{\infty}(\mathbb{C}), \overline{a_{\mu \nu}^{(r)}}= a_{\nu\mu}^{(r)}$ satisfying the following expansion with a sufficiently large integer $M^\star>0$:
		\begin{equation}\label{eq:R0-coef}
			a_{\mu \nu}^{(r)}(\xi)=\sum_{k=0}^{M^\star}\sum_{|\alpha+\beta|=2k}a^{(r)}_{\mu \nu \alpha \beta} \xi^{\alpha}\bar{\xi}^{\beta},
		\end{equation}
		(iii.1) we have
		$$
		\mathcal{R}_{1}^{(r)}=\sum_{\mu+\nu=2r+3} \xi^{\mu} \bar{\xi}^{\nu} \int_{\mathbb{R}^{3}} \mathbf{\Phi}_{\mu \nu}^{(r)}\left(x, \xi\right) \cdot  \mathbf{f}(x) d x
		$$
		where the map
		$$
		\mathbb{C}  \ni \xi \longmapsto \mathbf{\Phi}_{\mu \nu}^{(r)}(\cdot, \xi) \in\left(H^{n, s}\right)^{2} \text { is } C^{\infty}, \quad \mathbf{\Phi}_{\mu \nu}^{(r)}= (\Phi_{\mu \nu}^{(r)},\overline{\Phi_{\nu \mu}^{(r)}})
		$$
		satisfying the following expansion:
		\begin{equation}\label{eq:R1-coef}	
			\mathbf{\Phi}_{\mu \nu}^{(r)}(\cdot, \xi)=\sum_{k=0}^{M^\star}\sum_{|\alpha+\beta|=2k} \mathbf{\Phi}^{(r)}_{\mu \nu \alpha \beta}(x) \xi^{\alpha}\bar{\xi}^{\beta}
			%+\sum_{k=1}\sum_{|\alpha+\beta|=2k-1}b'_{\mu \nu \alpha\beta}\tilde{\Lambda}_{\mu \nu \alpha \beta}\xi^{\alpha}\bar{\xi}^{\beta}\int \tilde{\Phi}_{\mu\nu\alpha\beta}fdx
		\end{equation}
		with $\mathbf{\Phi}^{(r)}_{\mu \nu \alpha \beta}(x) \in \mathcal{S}\left(\mathbb{R}^{3}, \mathbb{C}\right)$.\\
		(iii.2-4) for $d=2,3,4$, we have
		\begin{equation}\label{eq:Rd}
			\mathcal{R}_{d}^{(r)}= \int_{\mathbb{R}^{3}} F_{d}^{(r)}\left(x, z\right)[U(x)]^{d} d x + \sum_{k} \prod_{l=1}^{d}\int_{\mathbb{R}^{3}} \mathbf{\Lambda}_{dlk}^{(r)}(x,z)\cdot\mathbf{f} d x,
		\end{equation}
		where $F_{4}^{(r)}\equiv 1;$ for $d=2,3,$
		%		the map
		%		$$
		%		\mathcal{U}^{-n,-s}  \ni z \longmapsto (F_{d}^{(r)}(\cdot, z), \mathbf{\Lambda}_{dlk}^{(r)}(\cdot, z)) \in (H^{n, s}\left(\mathbb{R}^{3}, \mathbb{C}\right))^3 \text { is } C^{\infty},
		%		$$
		$F_{d}^{(r)}(x, z) \in\mathbb{R}$ is a linear combination of terms of the form
		\begin{align}\label{eq:Rd-coef-F}
			\sum_{k= 0}^{M^\star}\sum_{i=0}^k\sum_{\mu+\nu= 4-d + 2k - i}  \xi^\mu\overline{\xi^\nu}\prod_{j=1}^i \int \mathbf{\Phi}^{ij}_{\mu\nu}(x) \cdot \mathbf{f} dx \Psi^{i}_{\mu\nu}(x),
		\end{align}
		and $\mathbf{\Lambda}_{dlk}^{(r)}(x, z)= (\Lambda_{dlk}^{(r)},\overline{\Lambda_{dlk}^{(r)}}) \ (d=2,3,4)$ is a linear combination of terms of the form
		\begin{align}\label{eq:Rd-coef-lambda}
			\sum_{k= 0}^{M^\star}\sum_{i=0}^k\sum_{\mu+\nu= 1 + 2k - i}  \xi^\mu\overline{\xi^\nu}\prod_{j=1}^i \int \tilde{\mathbf{\Phi}}^{ij}_{\mu\nu}(x) \cdot \mathbf{f} dx\tilde{\Psi}^{i}_{\mu\nu}(x),
		\end{align}
		with 
		$\mathbf{\Phi}^{ij}_{\mu\nu}(x),  \tilde{\mathbf{\Phi}}^{ij}_{\mu\nu}(x),
		\Psi^{i}_{\mu\nu}(x), \tilde{\Psi}^{i}_{\mu\nu}(x) \in \mathcal{S}\left(\mathbb{R}^{3}, \mathbb{C}\right),$ and the sum over index $k$ in \eqref{eq:Rd} is a finite sum.\\
		%\begin{align}\label{eq:Rd-coef}
		%	\left\|F_{d}^{(r)}(\cdot, z)\right\|_{H^{n, s}\left(\mathbb{R}^{3}, \mathbb{C}\right)} \lesssim C|\xi|^{4-d}, \quad \left\|\mathbf{\Lambda}_{dlk}^{(r)}(\cdot, z))\right\|_{H^{n, s}\left(\mathbb{R}^{3}, \mathbb{C}\right)} \lesssim |\xi|.
		%\end{align}
		\iffalse
		(iii.4) for $d=4$, we have
		$$
		\mathcal{R}_{4}^{(r)}= \int_{\mathbb{R}^{3}} [U(x)]^{4} d x + \sum_{k} \prod_{l=1}^{4}\int_{\mathbb{R}^{3}} \mathbf{\Lambda}_{4lk}^{(r)}(x,z)\cdot\mathbf{f} d x,
		$$
		where the map
		$$
		\mathcal{U}^{-n,-s}  \ni z \longmapsto  \mathbf{\Lambda}_{4lk}^{(r)}(\cdot, z)) \in (H^{n, s}\left(\mathbb{R}^{3}, \mathbb{C}\right))^2 \text { is } C^{\infty},
		$$ and $\mathbf{\Lambda}_{4lk}^{(r)}(x, z)$ is a linear combination of terms of the form
		\begin{align}\label{eq:R4-coef}
			\sum_{k= 0}^{M^\star}\sum_{i=0}^k\sum_{\mu+\nu= 1 + 2k - i} a^{(r)}_{i\mu\nu} \xi^\mu\overline{\xi^\nu}\prod_{j=1}^i \int \mathbf{\Phi}^{ij}_{\mu\nu}(x) \cdot \mathbf{f} dx,
		\end{align}
		\fi
		(iii.5) for $d=5$, we have
		$$\left\|\nabla_{z,\bar{z}}\mathcal{R}_{5}^{(r)}\right\|_{\left(\mathcal{P}^{n, s}\right)^2}\lesssim |\xi|^{M^\star}.$$
		
	\end{thm}
	\begin{rem}
		For $f\in X=W^{2,2}\left(\mathbb{R}^{3}, \mathbb{C}\right)\cap W^{2,1}\left(\mathbb{R}^{3}, \mathbb{C}\right),$ since $H^{k, s}\left(\mathbb{R}^{3}, \mathbb{C}\right) \subset X$ for $k, s$ large, we also have 
		$$
		\left\|(\xi, f)-\mathcal{T}_{r}(\xi, f)\right\|_{\mathbb{C} \times X} \lesssim \|(\xi, f)\|_{\mathcal{P}^{-n,-s}}^{3},
		$$ 
	\end{rem}
	\begin{proof}
		We prove Theorem \ref{thm:nft} by induction. We note that with some slightly abuses of notations, we denote $a$ with indexs as some constant, and denote $\Phi$ or $\Psi$ with indexs as some Schwartz function, they may change line from line, depending on the context. We also note that the sum with no upper index always denotes finite sum.\\
		\textbf{(Step 0)} First, when $r=0$, Theorem \ref{thm:nft} holds with $\mathcal{T}_{0}= I, Z^{(0)}=0, \mathcal{R}^{(0)}=H_P.$ And we have
		\begin{align*}
			R_{0}^{(0)}&=\sum_{\mu+\nu=4} \xi^{\mu} \bar{\xi}^{\nu} \int_{\mathbb{R}^{3}} \frac{\varphi^4}{2 \sqrt{\omega^{4} }}   d x,\\
			R_{1}^{(0)}&=\sum_{\mu+\nu=3} \xi^{\mu} \bar{\xi}^{\nu} \int_{\mathbb{R}^{3}} \frac{3\varphi^3}{2 \sqrt{\omega^{3}}} \left(B^{-\frac{1}{2}} f+B^{-\frac{1}{2}} \bar{f}\right) d x,\\
			R_{d}^{(0)}&=\int_{\mathbb{R}^{3}} F_{d}^{(0)} U^{d} d x, \quad F_{d}^{(0)}=\sum_{\mu+\nu=4-d} \frac{C_{4}^{d}}{\sqrt{2}^{4-d}\sqrt{\omega^{4-d}}} \xi^{\mu}\bar{\xi}^\nu \varphi^{4-d}  (d=2,3), F_4^{(0)}=1.  
		\end{align*}
		Thus, $a_{\mu \nu}^{(0)} \triangleq \int_{\mathbb{R}^{3}} \frac{\varphi^4}{2 \sqrt{\omega^{4}}}   d x$ and 
		$ 
		\mathbf{\Phi}_{\mu \nu}^{(0)} \triangleq \left(\frac{3}{2\sqrt{\omega^{3}}} B^{-\frac12}\left(\varphi^3 \right), \frac{3}{2 \sqrt{\omega^3}} B^{-\frac12}\left(\varphi^3\right) \right) .$
		
		\noindent \textbf{(Step $r\to r+1$)} Now we assume that the theorem holds for some $0\le r\le 2N$, we shall prove this for $r+1$. More precisely, define
		\begin{align*}
			&\mathcal{R}_{02}^{(r)}=\mathcal{R}_{0}^{(r)}-\sum_{\mu+\nu=2r+4}  a_{\mu \nu}^{(r)}( 0) \xi^{\mu} \bar{\xi}^{\nu}  ,\\
			&\mathcal{R}_{12}^{(r)}=\mathcal{R}_{1}^{(r)}-\sum_{\mu+\nu=2r+3} \xi^{\mu} \bar{\xi}^{\nu} \int_{\mathbb{R}^{3}} \mathbf{\Phi}_{\mu \nu}^{(r)}(x,0) \cdot \mathbf{f}(x) d x.
		\end{align*} 
		By \eqref{eq:R0-coef} and \eqref{eq:R1-coef}, we have
		\begin{align*}
			\mathcal{R}_{02}^{(r)} + \mathcal{R}_{12}^{(r)} = &\sum_{\mu+\nu=2r+6} a_{\mu \nu}^{(r+1)}(\xi) \xi^{\mu} \bar{\xi}^{\nu}   \\
			&+\sum_{\mu+\nu=2r+5} \xi^{\mu} \bar{\xi}^{\nu} \int_{\mathbb{R}^{3}} \mathbf{\Phi}_{\mu \nu}^{(r+1)}(x,\xi) \cdot \mathbf{f}(x) d x
			%&+ \int \mathbf{\Lambda}_{2,1}^{r+1}(x,z)\cdot \mathbf{f} d x \int  \mathbf{\Lambda}_{2,2}^{r+1}(x,z)\cdot \mathbf{f} d x,
		\end{align*} 
		where the coefficients $a_{\mu \nu}^{(r+1)}(\xi),\mathbf{\Phi}_{\mu \nu}^{(r+1)}(x,\xi)$ satisfy  \eqref{eq:R0-coef}-\eqref{eq:R1-coef} respectively, with $r$ replaced by $r+1$. 
		
		Set 
		\begin{align*}
			K_{r+1}:=& \sum_{\mu+\nu=2r+4} a_{\mu \nu}^{(r)}(0) \xi^{\mu} \bar{\xi}^{\nu} +\sum_{\mu+\nu=2r+3} \xi^{\mu} \bar{\xi}^{\nu} \int_{\mathbb{R}^{3}} \mathbf{\Phi}_{\mu \nu}^{(r)}(x,0) \cdot \mathbf{f}(x) d ,
		\end{align*}
		which is real-valued. Then, we solve the following homologic equation
		\begin{align*}
			\left\{H_{L}, \chi_{r+1}\right\}+Z_{r+1}=K_{r+1},
		\end{align*}
		with $Z_{r+1}$ in normal form. Thus, 
		\begin{align*}
			Z_{r+1}=&\sum_{\mu+\nu=2r+4, \mu=\nu 
			} a_{\mu \mu}^{(r)}(0) |\xi|^{2\mu} +\sum_{\substack{\mu+\nu=2r+3\\ \omega \cdot(\mu-\nu)<-m}} \xi^{\mu} \bar{\xi}^{\nu} \int \Phi_{\mu \nu}^{(r)}(x,0)  f(x) d x\\
			&+\sum_{\substack{\mu+\nu=2r+3\\\omega \cdot(\mu-\nu)>m}} \xi^{\mu} \bar{\xi}^{\nu} \int \overline{\Phi_{\nu\mu}^{(r)}}(x,0) \bar{f}(x) d x,
		\end{align*}
		and
		\begin{align*}
			\chi_{r+1}=& i\sum_{\substack{\mu+\nu=2r+4 \\
					\omega \cdot(\mu-\nu)\neq 0}}  \frac{a_{\mu \nu}^{(r)}(0)}{\omega \cdot(\mu-\nu)} \xi^{\mu} \bar{\xi}^{\nu} +i\sum_{\substack{\mu+\nu=2r+3\\ \omega \cdot(\mu-\nu)>-m}} \xi^{\mu} \bar{\xi}^{\nu} \int R_{\nu\mu}\Phi_{\mu \nu}^{(r)}(x,0)f d x\\
			&-i\sum_{\substack{\mu+\nu=2r+3\\\omega \cdot(\mu-\nu)<m}} \xi^{\mu} \bar{\xi}^{\nu} \int R_{\mu\nu}\overline{\Phi_{\nu\mu}^{(r)}}(x,0) \bar{f} d x,
		\end{align*}
		where the operator $$R_{\mu\nu} := (B-\omega \cdot(\mu-\nu))^{-1}.$$
		Let $\phi_{r+1}$ be the Lie transform generated by $\chi_{r+1}$, i.e. $\phi_{r+1}= \phi_{r+1}^t|_{t=1}$, where
		$$\frac{d\phi_{r+1}^t}{dt}= X_{\chi_{r+1}}=(-i\partial_{\bar{\xi}} \chi_{r+1},-i\nabla_{\bar{f}}\chi_{r+1}).$$
		Then, for $z'=(\xi',f')=\phi_{r+1}(\xi,f),$ Lemma  \ref{lem:z-expansion} holds, i.e.
		\begin{align}
			&\xi' = \xi + \sum_{k=1}^{\infty}\sum_{i=0}^{k}\sum_{\mu+\nu=(2r+2)k+1-i}a_{i\mu \nu} \xi^{\mu}\bar{\xi}^{\nu}\prod_{j=1}^{i}\int \mathbf{\Phi}_{\mu\nu}^{ij}\cdot \mathbf{f}dx, \label{eq:xi-expansion}\\
			&f'  = f+ \sum_{k=1}^{\infty}\sum_{i=0}^{k-1}\sum_{\mu+\nu=(2r+2)k+1-i} \xi^{\mu}\bar{\xi}^{\nu}\prod_{j=1}^{i}\int \mathbf{\Lambda}_{\mu\nu}^{ij}\cdot \mathbf{f}dx \Psi_{\mu, \nu}^{i}. \label{eq:f-expansion}
		\end{align}
		Recall that $$R^{(r)}=K_{r+1}+ R^{(r)}_{02}+ R^{(r)}_{12}+\sum_{d=2}^5 R^{(r)}_d,$$ and $ K_{r+1}= Z_{r+1} + \{H_L,\chi_{r+1}\}$, we have
		\begin{align}
			\nonumber H^{(r+1)}\triangleq & H^{(r)}\circ\phi_{r+1}= H\circ ( \mathcal{T}_{r}\circ\phi_{r+1})\equiv H\circ  \mathcal{T}_{r+1} \\
			= & H_L \circ \phi_{r+1}+ Z^{(r)} \circ \phi_{r+1} + R^{(r)}\circ \phi_{r+1} \nonumber \\
			= & H_L + Z^{(r)} + Z_{r+1} \nonumber\\
			&+[H_L\circ\phi_{r+1}-(H_L+\{\chi_{r+1},H_L\})] \label{eq:H-error-r}\\
			&+ Z^{(r)} \circ \phi_{r+1} - Z^{(r)} \label{eq:Z-error-r}\\
			&+(K_{r+1}\circ\phi_{r+1}-K_{r+1})\label{eq:K-error-r}\\
			&+ (R^{(r)}_{02}+ R^{(r)}_{12})\circ\phi_{r+1} \label{eq:R02-error}\\
			&+\sum_{d=2}^5 R^{(r)}_d\circ\phi_{r+1}.\label{eq:R-error-r}
		\end{align}
		We define $Z^{(r+1)}= Z^{(r)} + Z_{r+1}$ in the normal form of order $2r+4$. For the term \eqref{eq:H-error-r}, similar to the proof of Lemma \ref{lem:z-expansion}, we have
		\begin{align*}
			&H_L\circ\phi_{r+1}-(H_L+\{\chi_{r+1},H_L\})\\
			= &\sum_{k=2}^{\infty} \frac{1}{k!}\underbrace{\{\chi_{r+1},\dots\{\chi_{r+1}}_{k \text{ times}},H_L\}\} \\
			= &  \sum_{k=2}^{\infty}\sum_{i=0}^{k}\sum_{\mu+\nu=2(r+1)k+2-i}a_{i\mu \nu} \xi^{\mu}\bar{\xi}^{\nu}\prod_{j=1}^{i}\int \mathbf{\Phi}_{\mu\nu}^{ij}\cdot \mathbf{f}dx \\
			= & \sum_{k=2}^{M^*}\left(\sum_{\mu+\nu=2(r+1)k+2}a_{0\mu \nu} \xi^{\mu}\bar{\xi}^{\nu} + \sum_{\mu+\nu=2(r+1)k+1} a_{1\mu \nu}\xi^{\mu}\bar{\xi}^{\nu} \int  \mathbf{\Phi}^{11}_{\mu\nu}\cdot\mathbf{f}dx \right) \nonumber\\
			& + \sum_{k=2}^{M^*}\sum_{i=2}^{k}\sum_{\mu+\nu=2(r+1)k+2-i}a_{i\mu \nu} \xi^{\mu}\bar{\xi}^{\nu}\prod_{j=1}^{i}\int \mathbf{\Phi}_{\mu\nu}^{ij}\cdot \mathbf{f}dx + \mathcal{O}(|\xi|^{M^*}).
		\end{align*}
		Thus \eqref{eq:H-error-r} can be absorbed into $R^{(r+1)}_0, R^{(r+1)}_1$, $R^{(r+1)}_2$ and $R^{(r+1)}_5$.
		
		The terms \eqref{eq:Z-error-r},   \eqref{eq:K-error-r} and \eqref{eq:R02-error} can be handled similarly.
		
		For the term \eqref{eq:R-error-r}, denote $f'= f+ G_f, U'= U + G_U$, then for $d=2,3,4$, we have 
		\begin{align*}
			R^{(r)}_d\circ\phi_{r+1} &= \int_{\mathbb{R}^{3}} F_{d}^{(r)}\left(x, z'\right)(U+G_U)^{d} d x + \sum_{k}\prod_{l=1}^{d}\int_{\mathbb{R}^{3}} \mathbf{\Lambda}_{dlk}^{(r)}(x,z')\cdot \left(\mathbf{f} + \mathbf{G}_f\right)d x \\
			&= \sum_{j=0}^d \left[\int F_{d}^{(r)}\left(x, z'\right) U^{j} G_U^{d-j} d x + \sum_{k,l_i}\prod_{i=1}^{j}\int_{\mathbb{R}^{3}} \mathbf{\Lambda}_{d,l_i,k}^{(r)}(x,z')\cdot \mathbf{f} d x \prod_{l\neq l_i}\int_{\mathbb{R}^{3}} \mathbf{\Lambda}_{d,l}^{(r)}(x,z')\cdot \mathbf{G}_f d x\right]\\
			& := \sum_{j=0}^d H_{dj}.
		\end{align*}
		By \eqref{eq:f-expansion}, we have
		\begin{align*}
			G_f = \sum_{k=1}^{\infty}\sum_{i=0}^{k-1}\sum_{\mu+\nu=(2r+2)k+1-i} \xi^{\mu}\bar{\xi}^{\nu}\prod_{j=1}^{i}\int \mathbf{\Lambda}_{\mu\nu}^{ij}\cdot \mathbf{f}dx \Psi_{\mu, \nu}^{i}, \quad G_U= (G_f+\overline{G_f})/\sqrt{2B}.
		\end{align*}
		Therefore, by \eqref{eq:Rd-coef-F}, \eqref{eq:Rd-coef-lambda} and \eqref{eq:xi-expansion}, we derive
		\begin{align*}
			H_{d0} &= \int F_{d}^{(r)}\left(x, z'\right) G_U^{d} d x + \sum_{k}\prod_{l=1}^d\int_{\mathbb{R}^{3}} \mathbf{\Lambda}_{dlk}^{(r)}(x,z')\cdot \mathbf{G}_f d x \\
			&= \sum_{k=0}^{M^*}\sum_{i=0}^k \sum_{\mu+\nu=4+(2r+2)d + 2k -i} a_{i\mu\nu}\xi^{\mu}  \bar{\xi}^{\nu}\sum\prod_{j=1}^{i}\int \mathbf{\Phi}_{\mu\nu}^{ij}\cdot \mathbf{f}dx+\mathcal{O}(|\xi|^{M^*}),\\
			H_{d1} &= \int F_{d}^{(r)}\left(x, z'\right) U G_U^{d-1} d x +\sum_k \sum_{i=1}^d \int_{\mathbb{R}^{3}} \mathbf{\Lambda}_{dik}^{(r)}(x,z')\cdot \mathbf{f} d x \prod_{l\neq i}\int_{\mathbb{R}^{3}} \mathbf{\Lambda}_{dlk}^{(r)}(x,z')\cdot \mathbf{G}_f d x\\
			&= \sum_{k=0}^{M^*}\sum_{i=0}^k \sum_{\mu+\nu=3+(2r+2)(d-1) + 2k -i} a_{i\mu\nu}\xi^{\mu}  \bar{\xi}^{\nu}\sum\prod_{j=1}^{i+1}\int \mathbf{\Phi}_{\mu\nu}^{ij}\cdot \mathbf{f}dx+\mathcal{O}(|\xi|^{M^*}), 
		\end{align*}
		and for $2\le j \le d$,
		\begin{align*}
			H_{dj} &= \int F_{d}^{(r)}\left(x, z'\right) U^{j} G_U^{d-j} d x +\sum_{k,l_i} \prod_{i=1}^{j}\int_{\mathbb{R}^{3}} \mathbf{\Lambda}_{d,l_i,k}^{(r)}(x,z')\cdot \mathbf{f} d x \prod_{l\neq l_i}\int_{\mathbb{R}^{3}} \mathbf{\Lambda}_{dlk}^{(r)}(x,z')\cdot \mathbf{G}_f d x\\
			&= \int_{\mathbb{R}^{3}} F_{j}^{(r+1)}\left(x, z\right)U^{j} d x + \sum_k \prod_{l=1}^{j}\int_{\mathbb{R}^{3}} \mathbf{\Lambda}_{jlk}^{(r+1)}(x,z)\cdot \mathbf{f}d x + \mathcal{O}(|\xi|^{M^*})
		\end{align*}
		where 
		\begin{align*}
			F_{j}^{(r+1)}= F_{d}^{(r)}\left(x, z'\right) G_U^{d-j}-\mathcal{O}(|\xi|^{M^*}) = \sum_{k=0}^{M^*}\sum_{i=0}^k \sum_{\mu+\nu=4-j + (2r+2)(d-j) + 2k -i} a_{i\mu\nu}\xi^{\mu}  \bar{\xi}^{\nu}\sum\prod_{l=1}^{i}\int \mathbf{\Phi}_{\mu\nu}^{il}\cdot \mathbf{f}dx \psi^l_{\mu\nu}(x),
		\end{align*}
		note that $F_{4}^{(r+1)}\equiv 1.$
		Thus $H_{dj}$ can be absorbed into $R^{(r+1)}$. 
		Finally,  it is direct to see $R^{(r)}_5\circ\phi_{r+1}$ can be absorbed into $R^{(r+1)}_5.$
	\end{proof}

	\section{Isolation of the Key Resonant Terms}\label{Sec-isolation of the key resonant}
	Apply Theorem \ref{thm:nft} for $r=2N$, we obtain the new Hamiltonian 
	\begin{align*}
		H = H_L(\xi,\mathbf{f}) + Z_0(\xi) + Z_1(\xi,\mathbf{f}) + \mathcal{R},
	\end{align*}
	where
	$$ Z_1(\xi,\mathbf{f}) : = \langle G, f \rangle + \langle \bar{G}, \bar{f} \rangle, $$
	$$
	G:=\sum_{(\mu, \nu) \in M} \xi^{\mu} \bar{\xi}^{\nu} \Phi_{\mu \nu}(x), \Phi_{\mu \nu} \in \mathcal{S}\left(\mathbb{R}^{3}, \mathbb{C}\right),
	$$
	with 
	$$
	M=\{(\mu, \nu)\mid \mu +\nu=2r+1, 0 \leq r \leq 2 N, \omega(\mu-\nu)<-m\} .
	$$
	Then, the corresponding Hamiltonian equations are
	\begin{align}
		\dot{f} & = -\mathrm{i}(B f+\bar{G})-\mathrm{i}\partial_{\bar{f}}\mathcal{R}, \label{eq:f}\\
		\dot{\xi}& = - \mathrm{i}\omega\xi - \mathrm{i}\partial_{\bar{\xi}}Z_0-\mathrm{i}\left\langle \partial_{\bar{\xi}}G, f\right\rangle - \mathrm{i}\left\langle \partial_{\bar{\xi} }\bar{G}, \bar{f}\right\rangle -\mathrm{i}\partial_{\bar{\xi}}\mathcal{R}.\label{eq:xi}
	\end{align}
	%	Z_{1}(\xi, \xi, f, f):=\langle G, f\rangle+\langle G, f\rangle 
	%	The Hamilton equations of this system are given by
	%	$$
	%	\begin{aligned}
		%		\dot{f} &=-\mathrm{i}(B f+\bar{G}), \\
		%		\dot{\xi_{k}} &=-\mathrm{i} \omega_{k} \xi_{k}-\mathrm{i} \frac{\partial Z_{0}}{\partial \xi_{k}}-\mathrm{i}\left\langle\frac{\partial G}{\partial \xi_{k}}, f\right\rangle-\mathrm{i}\left\langle\frac{\partial \bar{G}}{\partial \xi_{k}}, \bar{f}\right\rangle .
		%	\end{aligned}
	%	$$
	Define
	$$
	M_1=\{(\mu, \nu)\mid \mu+\nu = 2N+1, \quad \omega(\mu-\nu)<-m\} ,
	$$ 
	and
	$$
	G=\sum_{(\mu, \nu) \in M} \xi^{\mu} \bar{\xi}^{\nu} \Phi_{\mu \nu}=\sum_{(\mu, \nu) \in M_1} \xi^{\mu} \bar{\xi}^{\nu} \Phi_{\mu \nu}+\sum_{(\mu, \nu) \in M\backslash M_1} \xi^{\mu} \bar{\xi}^{\nu} \Phi_{\mu \nu}:=\mathcal{M}_G+\mathcal{R}_G.
	$$
	Since $(2N-1)\omega<m<(2N+1)\omega$, we have 
	$$
	\mathcal{M}_G=\bar{\xi}^{2N+1}\Phi_{0,2N+1}.
	$$
	Let $\eta=e^{\mathrm{i}\omega t}\xi, $ by \eqref{eq:f} and Duhamel's formula, we have
	\begin{align}
		f(t) &=e^{-\mathrm{i}B t}f(0)+\int_{0}^{t}e^{-\mathrm{i}B (t-s)}(-\mathrm{i}\bar{G}-\mathrm{i}\partial_{\bar{f}}\mathcal{R})ds \label{f formula}\\
		&= e^{-\mathrm{i}B t}f(0) -\mathrm{i}e^{-\mathrm{i}B t}\sum_{(\mu, \nu) \in M}\int_{0}^{t}e^{\mathrm{i} (B-\omega (\nu-\mu))s}\eta^{\nu}\bar{\eta}^{\mu}\bar{\Phi}_{\mu \nu}ds -\mathrm{i}\int_{0}^{t}e^{-\mathrm{i}B (t-s)}\partial_{\bar{f}}\mathcal{R}ds \nonumber\\
		&:=\mathcal{M}_{f}+\mathcal{R}_{f}, \label{f expansion}
	\end{align}
	where
	\begin{align*}
		\mathcal{M}_{f}&=-\sum_{(\mu, \nu)\in M_1}e^{-\mathrm{i} \omega (\nu-\mu)t}\eta^{\nu}\bar{\eta}^{\mu}(B-\omega (\nu-\mu)-\mathrm{i}0)^{-1}\bar{\Phi}_{\mu \nu},\\
		&= 	-\xi^{2N+1}(B-(2N+1)\omega-i0)^{-1}\bar{\Phi}_{0,2N+1}
	\end{align*}
	is the main term and 
	\begin{align*}
		\mathcal{R}_{f}=&e^{-\mathrm{i}B t}f(0)-\mathrm{i}\int_{0}^{t}e^{-\mathrm{i}B (t-s)}\partial_{\bar{f}}\mathcal{R}ds\\
		&+\sum_{(\mu, \nu) \in M_1}\eta^{\nu}(0)\bar{\eta}^{\mu}(0)e^{-\mathrm{i}B t}(B-\omega (\nu-\mu)-\mathrm{i}0)^{-1}\bar{\Phi}_{\mu \nu} \\
		&+e^{-\mathrm{i}B t}\sum_{(\mu, \nu) \in M_1}\int_{0}^{t}e^{\mathrm{i} (B-\omega (\nu-\mu))s}\frac{d}{ds}(\eta^{\nu}\bar{\eta}^{\mu})(B-\omega (\nu-\mu)-\mathrm{i}0)^{-1}\bar{\Phi}_{\mu \nu}ds\\
		&-\mathrm{i}e^{-\mathrm{i}B t}\sum_{(\mu, \nu) \in M \backslash M_1}\int_{0}^{t}e^{\mathrm{i} (B-\omega (\nu-\mu))s}\eta^{\nu}\bar{\eta}^{\mu}\bar{\Phi}_{\mu \nu}ds
	\end{align*}
	are error terms, which will be estimated later, see Section \ref{Sec-error estimates}. Substitute \eqref{f expansion} into \eqref{eq:xi}, we obtain 
	\begin{equation}\label{eq:xi2}
		\dot{\xi} = - \mathrm{i}\omega\xi - \mathrm{i}\partial_{\bar{\xi}}Z_0-\mathrm{i}\left\langle \partial_{\bar{\xi}}\mathcal{M}_G, \mathcal{M}_{f}\right\rangle - \mathrm{i}\left\langle \partial_{\bar{\xi} }\bar{\mathcal{M}}_G, \bar{\mathcal{M}}_{f}\right\rangle+\mathcal{R}_{\xi},
	\end{equation}
	where 
	\begin{equation}\label{eq:Rxi}
		\mathcal{R}_{\xi}=-\mathrm{i}\left\langle \partial_{\bar{\xi}}\mathcal{R}_G, \mathcal{M}_{f}\right\rangle - \mathrm{i}\left\langle \partial_{\bar{\xi} }\bar{\mathcal{R}}_G, \bar{\mathcal{M}}_{f}\right\rangle-\mathrm{i}\left\langle \partial_{\bar{\xi}}G, \mathcal{R}_{f}\right\rangle - \mathrm{i}\left\langle \partial_{\bar{\xi} }\bar{G}, \bar{\mathcal{R}}_{f}\right\rangle -\mathrm{i}\partial_{\bar{\xi}}\mathcal{R}
	\end{equation}
	can be treated as error terms. Hence, we compute
	\begin{align*}
		\frac{d}{dt}|\xi|^2 & =2\operatorname{Re}(\bar{\xi} \dot{\xi})\\
		&=-2(2N+1)|\xi|^{4N+2}\operatorname{Im}\langle \Phi_{0,2N+1},  (B-(2N+1)\omega-i0)^{-1}\bar{\Phi}_{0,2N+1}\rangle+2Re(\bar{\xi}\mathcal{R}_{\xi}).
	\end{align*}
	Note that by using Plemelji formula
	$$\frac{1}{x \mp i 0}= \operatorname{P.V} \frac{1}{x} \pm \mathrm{i} \pi \delta(x),$$ 
	we have
	\begin{align*}
		\gamma :=& (2N+1)\operatorname{Im}\langle \Phi_{0,2N+1},  (B-(2N+1)\omega-i0)^{-1}\bar{\Phi}_{0,2N+1}\rangle\\
		=& (2N+1)\langle \Phi_{0,2N+1},  \delta(B-(2N+1)\omega)\bar{\Phi}_{0,2N+1}\rangle\\
		\ge& 0,
	\end{align*}
	and
	\begin{equation}\label{eq:xi3}
		\frac{d}{dt}|\xi|^2=-2\gamma |\xi|^{4N+2}+2Re(\bar{\xi}\mathcal{R}_{\xi}).
	\end{equation}
	Throughout this paper, we assume the following non-degenerate assumption.
	\begin{assu}[Fermi's Golden Rule]
		$\gamma >0.$
	\end{assu}
	\begin{rem}
		The Fermi's Golden Rule condition implies that $|\xi|\approx \frac{|\xi_0|}{(1+|\xi_0|^{4N} t)^{\frac{1}{4N}}}$ if we could neglect the error term $2Re(\bar{\xi}\mathcal{R}_{\xi})$, which will be justified in Section \ref{Sec-error estimates} .
	\end{rem}
	
	\section{Asymptotic Behavior}\label{Sec-aymptotic behavior}
	In this section, we derive the asymptotic behavior of $\xi$  and $f$. Before proceeding, we introduce some useful notations. Let $T>0$ be fixed, denote $\langle t\rangle=(1+t^2)^{1/2}$ and 
	$$\begin{gathered}
		{[\xi]_{\frac{1}{4N}}(T)=\sup _{0 \leq t \leq T}\langle |\xi_0|^{4N}t\rangle^{\frac{1}{4N}}|\xi(t)|},
	\end{gathered}$$
	We define the norm $\|\cdot\|_{X}$ as
	$$\|f\|_{X}= \|f\|_{W^{2,2}} + \|f\|_{W^{2,1}}. $$
	
	\subsection{Dynamics of $\xi$}	
	The following theorem allows to treat $\mathcal{R}_{\xi}$ perturbatively in the dynamics of $\xi$:
	\begin{thm}\label{ode}
		Suppose that $\gamma >0$ and the error term $\mathcal{R}_{\xi}$ satisfies $$
		|\mathcal{R}_{\xi}(t)|\le Q_0 (1+4N\gamma|\xi_0|^{4N} t) ^{-\frac{4N+1}{4N}-\delta}$$
		for some small constant $Q_0$ and $\delta >0$, then
		\begin{equation}
			|\xi(t)|^{4N}\le (1+4N\gamma|\xi_0|^{4N} t)^{-1}\left(|\xi_0|^{4N}+\frac{C(\gamma^{-1/2}Q_0^{1/2}|\xi_0|^{\frac{4N-1}{2}}+\gamma^{-\frac{4N}{4N+1}}Q_0^{\frac{4N}{4N+1}})}{(1+4N\gamma|\xi_0|^{4N} t)^{\delta/2}}\right)
		\end{equation}
		for some absolute constant $C$.
		In addition, if $Q_0=O(|\xi_0|^{4N+1+\epsilon})$ and $|\xi_0|$ is sufficiently small, then
		\begin{equation}
			|\xi(t)|^{4N}\ge (1+4N\gamma|\xi_0|^{4N} t)^{-1}\left(|\xi_0|^{4N}-\frac{|\xi_0|^{4N+\epsilon}}{(1+4N\gamma|\xi_0|^{4N} t)^{\delta}}\right).	
		\end{equation}
	\end{thm}
	\begin{proof}
		We use standard comparison theorem to prove Theorem \ref{ode}. Define $r=|\xi|^{4N}$, by \eqref{eq:xi3} we have
		\begin{equation}\label{r'}
			-4N\gamma r^2-4N|\mathcal{R}_{\xi}|r^{\frac{4N-1}{4N}}	\le r'\le -4N\gamma r^2+4N|\mathcal{R}_{\xi}|r^{\frac{4N-1}{4N}},  	
		\end{equation}		
		Define $$h(t)=(1+4N\gamma r_0 t)^{-1}(r_0+C_0(1+4N\gamma r_0 t)^{-\delta/2}),$$ then
		$$h'(t)=-\frac{4N\gamma r_0^2}{(1+4N\gamma r_0t)^2}-\frac{4N(1+\delta/2)\gamma r_0C_0}{(1+4N\gamma r_0t)^{2+\delta/2}}.$$
		Hence,
		\begin{align*}
			h'+4N\gamma h^2&=\frac{4N(1-\delta/2)\gamma r_0C_0}{(1+4N\gamma r_0t)^{2+\delta/2}}+\frac{4N\gamma C_0^2}{(1+4N\gamma r_0t)^{2+\delta}}\nonumber\\ 
			&\ge \frac{4N\gamma C_0^2}{(1+4N\gamma r_0t)^{2+\delta}}.
		\end{align*}
		Note that 
		\begin{equation*}
			h^{\frac{4N-1}{4N}}\le (1+4N\gamma r_0 t)^{-\frac{4N-1}{4N}}(r_0^{\frac{4N-1}{4N}}+C_0^{\frac{4N-1}{4N}}(1+4N\gamma r_0 t)^{-\frac{(4N-1)\delta}{8N}}),
		\end{equation*} 
		then 
		\begin{equation*}
			4N|\mathcal{R}_{\xi}|h^{\frac{4N-1}{4N}}\le \frac{4N Q_0r_0^{\frac{4N-1}{4N}}+4N Q_0C_0^{\frac{4N-1}{4N}}}{(1+4N\gamma r_0t)^{2+\delta}}.
		\end{equation*}
		Hence, if we choose $C_0$  such that
		\begin{equation}
			4N\gamma C_0^2\ge 8N Q_0r_0^{\frac{4N-1}{4N}}+8N Q_0C_0^{\frac{4N-1}{4N}},\label{C_0}
		\end{equation}
		then 
		\begin{equation}\label{h'}
			h'+4N\gamma h^2> 4N|\mathcal{R}_{\xi}|h^{\frac{4N-1}{4N}}.
		\end{equation} 	
		By comparing \eqref{r'} and \eqref{h'} and note that $r_0\le h_0$, we get $r(t)\le h(t)$ and the desired estimate is proved. To achieve the condition \eqref{C_0}, it suffices to  choose $C_0=C(\gamma^{-1/2}Q_0^{1/2}r_0^{\frac{4N-1}{8N}}+\gamma^{-\frac{4N}{4N+1}}Q_0^{\frac{4N}{4N+1}})$ for some large constant $C$.
		
		For the lower bound of $|\xi|$, note that
		\begin{equation}
			r'\ge -4N\gamma r^2-4N|\mathcal{R}_{\xi}|r^{\frac{4N-1}{4N}}.
		\end{equation}
		Choose 
		$$
		\tilde{h}(t)=(1+4Nr_0\gamma t)^{-1}\left(r_0-r_0^{(1+\epsilon/8N)}(1+4Nr_0\gamma t)^{-\delta}\right),
		$$
		similarly we compute
		\begin{align*}
			\tilde{h}'+4N\gamma \tilde{h}^2&=-\frac{4N\gamma(1-\delta) r_0^{2+\frac{\epsilon}{8N}}}{(1+4N\gamma r_0t)^{2+\delta}}+\frac{4N\gamma r_0^{2+\frac{\epsilon}{4N}}}{(1+4N\gamma r_0t)^{2+2\delta}}\nonumber\\ 
			&\lesssim -\frac{ r_0^{2+\frac{\epsilon}{8N}}}{(1+4N\gamma r_0t)^{2+\delta}},
		\end{align*}
		and
		\begin{align*}
			\tilde{h}^{\frac{4N-1}{4N}}< (1+4N\gamma r_0 t)^{-\frac{4N-1}{4N}}r_0^{\frac{4N-1}{4N}},
		\end{align*}
		hence
		\begin{align*}
			4N|\mathcal{R}_{\xi}|\tilde{h}^{\frac{4N-1}{4N}}\lesssim \frac{ r_0^{2+\frac{\epsilon}{4N}}}{(1+4N\gamma r_0t)^{2+\delta}}.
		\end{align*}
		Therefore if we choose $|\xi_0|$ to be sufficiently small, then 
		\begin{equation}
			\tilde{h}'<-4N\gamma \tilde{h}^2-4N|\mathcal{R}_{\xi}|\tilde{h}^{\frac{4N-1}{4N}}.
		\end{equation}
		Using comparison theorem we get $r\ge \tilde{h}.$
	\end{proof}	
	
	\subsection{Dynamics of $f$}
	By \eqref{f formula}, $f$ can be decomposed as 
	\begin{align*}
		f(t) &=e^{-\mathrm{i}B t}f(0)+\int_{0}^{t}e^{-\mathrm{i}B (t-s)}(-\mathrm{i}\bar{G}-\mathrm{i}\partial_{\bar{f}}\mathcal{R})ds.
	\end{align*}
	Denote
	$$\hat{\mathcal{R}}_f :=  -\mathrm{i} \int_{0}^{t}e^{-\mathrm{i}B (t-s)}\partial_{\bar{f}}\mathcal{R}ds.$$
	\begin{prop}\label{est:f}	
		Let $\sigma> \frac52$, suppose that the error term $\hat{\mathcal{R}}_f$ satisfies 
		\begin{align}
			&\|B^{-1/2} \hat{\mathcal{R}}_f\|_{L^8}\lesssim \langle t \rangle^{-\frac{2N+1}{4N}}\mathcal{P}_1
			+\langle |\xi_0|^{4N}t \rangle^{-\frac{2N+1}{4N}}\mathcal{P}_2
			, \label{assumption RfL8}\\
			&\|\langle x \rangle^{-\sigma}B^{1/2}\hat{\mathcal{R}}_f\|_{L^4}\lesssim \langle t \rangle^{-\frac{1}{2}}\mathcal{P}_3+\langle |\xi_0|^{4N}t \rangle^{-\frac{1}{2}}\mathcal{P}_4, \label{assumption RfL4}
		\end{align}
		for some constants $\mathcal{P}_1,\mathcal{P}_2,\mathcal{P}_3,\mathcal{P}_4$, then the following estimates hold
		\begin{align*}
			&\|B^{-1/2}f\|_{L^8}\lesssim \langle t \rangle^{-\frac{2N+1}{4N}}\Big(\|B^{-1/2}f(0)\|_X+\mathcal{P}_1\Big)
			+\langle |\xi_0|^{4N}t \rangle^{-\frac{2N+1}{4N}}\Big([\xi]^{2N+1}_{\frac{1}{4N}}+\mathcal{P}_2\Big)
			,\\
			&\|\langle x \rangle^{-\sigma}B^{1/2}f\|_{L^4}\lesssim \langle t \rangle^{-\frac{1}{2}}\Big(\|B^{-1/2}f(0)\|_X+\mathcal{P}_3\Big)+\langle |\xi_0|^{4N}t \rangle^{-\frac{1}{2}}\Big([\xi]^{2N+1}_{\frac{1}{4N}}+\mathcal{P}_4\Big).
		\end{align*}
	\end{prop}
	
	\begin{proof}
		By Lemma \ref{Lp-Lp'} (choosing $l=0, p'=8, \theta=1$), we have
		\begin{align*}
			\|B^{-1/2}e^{-\mathrm{i}B t}f(0)\|_{L^8} \lesssim \langle t\rangle^{-9/8}\|B^{-1/2}f(0)\|_{W^{2,\frac87}}\lesssim \langle t\rangle^{-9/8}\|B^{-1/2}f(0)\|_{X}. 
		\end{align*}
		Similarly, (choosing $l=0, p'=4, \theta=0$)
		\begin{align*}
			\|B^{1/2}e^{-\mathrm{i}B t}f(0)\|_{L^4} \lesssim \langle t\rangle^{-1/2}\|B^{1/2}f(0)\|_{W^{1,\frac43}}\lesssim \langle t\rangle^{-1/2}\|B^{-1/2}f(0)\|_{X}. 
		\end{align*}
		Thus, it suffices to prove that 
		\begin{align}\label{esti:G1}
			\left\|B^{-1/2}\int_{0}^{t}e^{-\mathrm{i}B (t-s)}\bar{G}ds\right\|_{L^8} \lesssim \langle |\xi_0|^{4N}t \rangle^{-\frac{2N+1}{4N}}[\xi]^{2N+1}_{\frac{1}{4N}},
		\end{align}
		and
		\begin{align}\label{esti:G2}
			\left\|\langle x \rangle^{-\sigma}B^{1/2}\int_{0}^{t}e^{-\mathrm{i}B (t-s)}\bar{G}ds\right\|_{L^4} \lesssim \langle |\xi_0|^{4N}t \rangle^{-\frac{2N+1}{4N}}[\xi]^{2N+1}_{\frac{1}{4N}}.
		\end{align}
		Note that by definition,
		$$
		G=\sum_{(\mu, \nu) \in M} \xi^{\mu} \bar{\xi}^{\nu} \Phi_{\mu \nu}(x), \Phi_{\mu \nu} \in \mathcal{S}\left(\mathbb{R}^{3}, \mathbb{C}\right),
		$$
		with 
		$
		M=\{(\mu, \nu)\mid \mu +\nu=2r+1, 0 \leq r \leq 2 N, \omega(\mu-\nu)<-m\} .$ Since $(2N-1)\omega<m<(2N+1)\omega$, we have 
		$$\mu+\nu\ge 2N+1,$$
		which implies \eqref{esti:G1} and \eqref{esti:G2} by using Lemma \ref{convo1} in Appendix.
	\end{proof}
	
	\section{Error Estimates}\label{Sec-error estimates}
	In this section, we estimate error terms $\hat{\mathcal{R}}_{f}$ and $\mathcal{R}_{\xi}$ by using the asymptotic behavior of $\xi$ and $f$. Once this was achieved, we can use bootstrap arguments to finish our proof, see Section \ref{Sec-proof of main result}. To proceed, we assume that there exist positive constants $A_f(T), B_f(T), C_f(T), D_f(T)$ such that for $0\le t\le T,$ it holds that
	\begin{align}
		&\|B^{-1/2}f\|_{L^8}\lesssim \langle t \rangle^{-\frac{2N+1}{4N}}A_f(T)
		+\langle |\xi_0|^{4N}t \rangle^{-\frac{2N+1}{4N}}B_f(T),\label{assumption f1}\\
		&\|\langle x \rangle^{-\sigma}B^{1/2}f\|_{L^4}\lesssim \langle t \rangle^{-\frac{1}{2}}C_f(T)+\langle |\xi_0|^{4N}t \rangle^{-\frac{1}{2}}D_f(T).\label{assumption f2}
	\end{align}
	%\begin{rem}
	%	We assume $[\xi]_{\frac{1}{4N}}(T), A_f(T), B_f(T), C_f(T), D_f(T)$ to be small (i.e. less that $1$) for convenience of subsequent estimates. 	
	%\end{rem}
	Recall that
	$$\hat{\mathcal{R}}_f =  -\mathrm{i} \int_{0}^{t}e^{-\mathrm{i}B (t-s)}\partial_{\bar{f}}\mathcal{R}ds,$$
	we first prove estimates of $\partial_{\bar{f}}\mathcal{R}=\sum_{d=1}^{5} \partial_{\bar{f}}\mathcal{R}_{d}.$
	%	\begin{equation}
		%	\|B^{-1/2}\partial_{\bar{f}}\mathcal{R}\|_{W^{2,8/7}}\lesssim \sum_{d=1}^{5}\|B^{-1/2}\partial_{\bar{f}}\mathcal{R}_{d}\|_{W^{2,8/7}}
		%	\end{equation}
	\begin{lem}\label{est:partial f of R 1}The following estimates hold:
		\begin{align*}
			\|B^{-1/2}\partial_{\bar{f}}\mathcal{R}_{1}\|_{W^{2,8/7}}&\lesssim \langle |\xi_0|^{4N}t \rangle^{-\frac{4N+3}{4N}}[\xi]^{4N+3}_{\frac{1}{4N}},  \\
			\|B^{-1/2}\partial_{\bar{f}}\mathcal{R}_{2}\|_{W^{2,8/7}}&\lesssim \langle |\xi_0|^{4N}t \rangle^{-\frac{2}{4N}}[\xi]^{2}_{\frac{1}{4N}}\Big(\|B^{-1/2}f\|_{L^8}+\|\langle x \rangle^{-\sigma}B^{1/2}f\|_{L^4}\Big) ,\\
			\|B^{-1/2}\partial_{\bar{f}}\mathcal{R}_{3}\|_{W^{2,8/7}}&\lesssim \langle |\xi_0|^{4N}t \rangle^{-\frac{1}{4N}} [\xi]_{\frac{1}{4N}}\|B^{-1/2}f\|_{L^8}\Big(\|B^{-1/2}f\|_{L^8}+\|\langle x \rangle^{-\sigma}B^{1/2}f\|_{L^4}\Big),\\
			\|B^{-1/2}\partial_{\bar{f}}\mathcal{R}_{4}\|_{W^{2,8/7}}&\lesssim  \|B^{1/2}f\|_{L^2}^{4/3}\|B^{-1/2}f\|_{L^8}^{5/3}\\
			\|B^{-1/2}\partial_{\bar{f}}\mathcal{R}_{5}\|_{W^{2,8/7}}&\lesssim
			\langle |\xi_0|^{4N}t \rangle^{-\frac{M^\star}{4N}}[\xi]^{M^\star}_{\frac{1}{4N}} .
		\end{align*}
	\end{lem}
	
	\begin{proof}
		The estimates of $\mathcal{R}_{1}$ and $\mathcal{R}_{5}$ are direct. For $\mathcal{R}_{d}, d=2,3,$ by Theorem \ref{thm:nft}, $F^{(2N)}_d(x,z),$ are linear combinations of terms of the form $\xi^{\mu}\bar{\xi}^{\nu}\prod_j^i\int \mathbf{\Phi}_{j}\cdot\mathbf{f}dx\Psi,$
		where $\mu+\nu\ge 4-d+2k-i, 0\le i\le k, \mathbf{\Phi}_{j},\Psi \in \mathcal{S}\left(\mathbb{R}^{3}, \mathbb{C}\right).$ And  $ \Lambda^{(2N)}_{dlk}(x,z)$ are linear combinations of terms of the form $\xi^{\mu}\bar{\xi}^{\nu}\prod_j^i\int \mathbf{\Phi}_{j}\cdot\mathbf{f}dx\Psi,$
		where $\mu+\nu\ge 1+2k-i, 0\le i\le k, \mathbf{\Phi}_{j},\Psi \in \mathcal{S}\left(\mathbb{R}^{3}, \mathbb{C}\right).$ Hence, $\partial_{\bar{f}}\mathcal{R}_{d}$ are linear combinations of terms of forms		
		$$\xi^{\mu}\bar{\xi}^{\nu}\prod_j^i\int \mathbf{\Phi}_{j} \cdot\mathbf{f}dx \int \Psi U^{d}dx \Psi',\quad  \mu+\nu\ge 4-d+2(k+1)-i, 0\le i\le k, $$
		or
		$$\xi^{\mu}\bar{\xi}^{\nu}\prod_j^i\int \mathbf{\Phi}_{j} \cdot\mathbf{f}dx B^{-1/2}\left(\Psi U^{d-1}\right), \quad  \mu+\nu\ge 4-d+2k-i, 0\le i\le k,$$
		or
		$$\xi^{\mu}\bar{\xi}^{\nu}\prod_j^{d-1+i}\int \mathbf{\Phi}_{j} \cdot\mathbf{f}dx \Psi, \quad  \mu+\nu\ge 4-d+2k-i, 0 \le i\le k.$$
		In each form, $\mu+\nu\ge 4-d$ and $f$ has at least order $d-1.$ Then the rest follows directly by H\"older's inequality and Leibnitz rule. For $d=4,$ we have
		\begin{align*}
			&\|B^{-1}\left(\left(B^{-1/2}f+B^{-1/2}\bar{f}\right)^3\right)\|_{W^{2,8/7}}\\
			\lesssim &	\|\left(B^{-1/2}f+B^{-1/2}\bar{f}\right)^3\|_{W^{1,8/7}}\\
			\lesssim &   \|B^{-1/2}f\|_{W^{1,2}}^{4/3}\|B^{-1/2}f\|_{L^8}^{5/3}.
		\end{align*}	
	\end{proof}
	\begin{cor}
		Assuming \eqref{assumption f1} and \eqref{assumption f2}, it holds that
		\begin{align*}
			\|B^{-1/2}\partial_{\bar{f}}\mathcal{R}_{1}\|_{W^{2,8/7}}&\lesssim \langle |\xi_0|^{4N}t \rangle^{-\frac{4N+3}{4N}}[\xi]^{4N+3}_{\frac{1}{4N}},  \\
			\|B^{-1/2}\partial_{\bar{f}}\mathcal{R}_{2}\|_{W^{2,8/7}}&\lesssim \langle |\xi_0|^{4N}t \rangle^{-\frac{1}{4N}}\langle t \rangle^{-\frac{2N+1}{4N}}[\xi]^{2}_{\frac{1}{4N}}\Big(A_f+|\xi_{0}|^{-1}C_f\Big)+\langle |\xi_0|^{4N}t \rangle^{-\frac{2N+2}{4N}}[\xi]^{2}_{\frac{1}{4N}}\big(B_f+D_f\big) ,\\
			\|B^{-1/2}\partial_{\bar{f}}\mathcal{R}_{3}\|_{W^{2,8/7}}&\lesssim \langle |\xi_0|^{4N}t \rangle^{-\frac{1}{4N}}\langle t \rangle^{-\frac{2N+1}{4N}}[\xi]_{\frac{1}{4N}}A_f\big(A_f+C_f+D_f\big)+\langle |\xi_0|^{4N}t \rangle^{-\frac{2N+2}{4N}}[\xi]_{\frac{1}{4N}}B_f\big(B_f+C_f+D_f\big),\\
			\|B^{-1/2}\partial_{\bar{f}}\mathcal{R}_{4}\|_{W^{2,8/7}}&\lesssim  \langle t \rangle^{-\frac{2N+2}{4N}}\|B^{1/2}f\|_{L^2}^{4/3}A_f^{5/3}+\langle |\xi_0|^{4N}t \rangle^{-\frac{2N+2}{4N}}\|B^{1/2}f\|_{L^2}^{4/3}B_f^{5/3},\\
			\|B^{-1/2}\partial_{\bar{f}}\mathcal{R}_{5}\|_{W^{2,8/7}}&\lesssim
			\langle |\xi_0|^{4N}t \rangle^{-\frac{M^\star}{4N}}[\xi]^{M^\star}_{\frac{1}{4N}} .
		\end{align*}
	\end{cor}
	\begin{proof}
		The proof is direct by substituting \eqref{assumption f1} and \eqref{assumption f2} into Lemma \ref{est:partial f of R 1}, we have used the fact that $\langle |\xi_0|^{4N}t \rangle^{-\frac{1}{4N}}|\xi_{0}|\le \langle t \rangle^{-\frac{1}{4N}}$.
	\end{proof}
	Similarly, we have the following lemma. The proof also relies on the explicit forms of $\partial_{\bar{f}}\mathcal{R}_{d}$ and standard H\"older's inequality. Here we omit the details.
	\begin{lem}\label{est:partial f of R 2} Let $\sigma> \frac52$, the following estimates hold:
		\begin{align*}
			&\|\langle x \rangle^{\sigma}B^{1/2}\partial_{\bar{f}}\mathcal{R}_{1}\|_{W^{1,4/3}}\lesssim \langle |\xi_0|^{4N}t \rangle^{-\frac{4N+3}{4N}}[\xi]^{4N+3}_{\frac{1}{4N}}  ,\\
			&\|\langle x \rangle^{\sigma}B^{1/2}\partial_{\bar{f}}\mathcal{R}_{2}\|_{W^{1,4/3}}\lesssim\langle |\xi_0|^{4N}t \rangle^{-\frac{2}{4N}}[\xi]^{2}_{\frac{1}{4N}}\Big(\|B^{-1/2}f\|_{L^8}+\|\langle x \rangle^{-\sigma}B^{1/2}f\|_{L^4}\Big) ,\\
			&\|\langle x \rangle^{\sigma}B^{1/2}\partial_{\bar{f}}\mathcal{R}_{3}\|_{W^{1,4/3}}\lesssim \langle |\xi_0|^{4N}t \rangle^{-\frac{1}{4N}} [\xi]_{\frac{1}{4N}}\|B^{-1/2}f\|_{L^8}\Big(\|B^{-1/2}f\|_{L^8}+\|\langle x \rangle^{-\sigma}B^{1/2}f\|_{L^4}\Big) ,\\
			&\|B^{1/2}\partial_{\bar{f}}\mathcal{R}_{4}\|_{W^{1,4/3}}\lesssim \|B^{1/2}f\|_{L^2}\|B^{-1/2}f\|_{L^8}^{2},\\
			&\|\langle x \rangle^{\sigma}B^{1/2}\partial_{\bar{f}}\mathcal{R}_{5}\|_{W^{1,4/3}}\lesssim
			\langle |\xi_0|^{4N}t \rangle^{-\frac{M^\star}{4N}}[\xi]^{M^\star}_{\frac{1}{4N}},
		\end{align*}
	\end{lem}
	\begin{cor} Let $\sigma> \frac52$ and assume \eqref{assumption f1} and \eqref{assumption f2} hold , then 
		\begin{align*}
			&\|\langle x \rangle^{\sigma}B^{1/2}\partial_{\bar{f}}\mathcal{R}_{1}\|_{W^{1,4/3}}\lesssim \langle |\xi_0|^{4N}t \rangle^{-\frac{4N+3}{4N}}[\xi]^{4N+3}_{\frac{1}{4N}}  ,\\
			&\|\langle x \rangle^{\sigma}B^{1/2}\partial_{\bar{f}}\mathcal{R}_{2}\|_{W^{1,4/3}}\lesssim\langle t \rangle^{-\frac{1}{2}}[\xi]^{2}_{\frac{1}{4N}}\big(A_f+C_f\big)+\langle |\xi_0|^{4N}t \rangle^{-\frac{2N+2}{4N}}[\xi]^{2}_{\frac{1}{4N}}\big(B_f+D_f\big)  ,\\
			&\|\langle x \rangle^{\sigma}B^{1/2}\partial_{\bar{f}}\mathcal{R}_{3}\|_{W^{1,4/3}}\lesssim \langle t \rangle^{-\frac{2N+1}{4N}}[\xi]_{\frac{1}{4N}}A_f\big(A_f+C_f+D_f\big)+\langle |\xi_0|^{4N}t \rangle^{-\frac{2N+2}{4N}}[\xi]_{\frac{1}{4N}}B_f\big(B_f+C_f+D_f\big) ,\\
			&\|B^{1/2}\partial_{\bar{f}}\mathcal{R}_{4}\|_{W^{1,4/3}}\lesssim \langle t \rangle^{-\frac{2N+2}{4N}}\|B^{1/2}f\|_{L^2}A_f^{2}+\langle |\xi_0|^{4N}t \rangle^{-\frac{2N+2}{4N}}\|B^{1/2}f\|_{L^2}B_f^{2},\\
			&\|\langle x \rangle^{\sigma}B^{1/2}\partial_{\bar{f}}\mathcal{R}_{5}\|_{W^{1,4/3}}\lesssim
			\langle |\xi_0|^{4N}t \rangle^{-\frac{M^\star}{4N}}[\xi]^{M^\star}_{\frac{1}{4N}},
		\end{align*}
	\end{cor}
	Now we can derive estimates of $\hat{\mathcal{R}}_{f}$ and $\mathcal{R}_{\xi}$:
	\begin{prop}\label{thm:hat-Rf} For $\sigma> \frac52$, the following estimates hold:
		\begin{align*}
			&\left\|B^{-1/2}\hat{\mathcal{R}}_f\right\|_{L^8}  \lesssim \langle t \rangle^{-\frac{2N+1}{4N}}\mathcal{P}_1+
			\langle |\xi_0|^{4N}t \rangle^{-\frac{2N+2}{4N}}\mathcal{P}_2,\\
			&\left\| \langle x \rangle^{-\sigma}B^{1/2} \hat{\mathcal{R}}_f \right\|_{L^4} \lesssim \langle t \rangle^{-\frac{1}{2}} \mathcal{P}_3+\langle |\xi_0|^{4N}t \rangle^{-\frac{1}{2}} \mathcal{P}_4,
		\end{align*}
		where $\mathcal{P}_1, \mathcal{P}_2, \mathcal{P}_3, \mathcal{P}_4$ are  defined by \eqref{p1 formula}, \eqref{p2 formula}, \eqref{p3 formula} and \eqref{p4 formula} respectively.
	\end{prop}
	\begin{proof}	
		By Lemma \ref{Lp-Lp'}, Lemma \ref{est:partial f of R 1} and the corollary thereafter, we have
		\begin{align*}
			\left\|B^{-1/2}\hat{\mathcal{R}}_f\right\|_{L^8}\lesssim & \int_{0}^{t}\left\|e^{-\mathrm{i}B (t-s)}B^{-1/2}\partial_{\bar{f}}\mathcal{R}\right \|_{L^8}ds \\
			\lesssim& \int_{0}^{t}\min\{|t-s|^{-9 / 8},|t-s|^{-3 / 8}\}\|B^{-1/2}\partial_{\bar{f}}\mathcal{R}\|_{W^{2,8/7}}ds\\
			\lesssim & \int_{0}^{t}\min\{|t-s|^{-9 / 8},|t-s|^{-3 / 8}\}
			\Big(\langle s \rangle^{-\frac{2N+1}{4N}}\langle |\xi_0|^{4N}s \rangle^{-\frac{1}{4N}}\mathcal{P}_1+
			\langle |\xi_0|^{4N}s \rangle^{-\frac{2N+2}{4N}}\mathcal{P}_2\Big)ds\\
			\lesssim & \langle t \rangle^{-\frac{2N+1}{4N}}\langle |\xi_0|^{4N}t \rangle^{-\frac{1}{4N}}\mathcal{P}_1+
			\langle |\xi_0|^{4N}t \rangle^{-\frac{2N+2}{4N}}\mathcal{P}_2,
		\end{align*}
		where we can choose
		\begin{align}
			\mathcal{P}_1=&[\xi]^{2}_{\frac{1}{4N}}A_f+[\xi]^{2}_{\frac{1}{4N}}|\xi_{0}|^{-1}C_f+[\xi]_{\frac{1}{4N}}A_f^2+[\xi]_{\frac{1}{4N}}A_fC_f+[\xi]_{\frac{1}{4N}}A_fD_f+	|\xi_{0}|^{4/3}A_f^{5/3},\label{p1 formula}\\	
			\mathcal{P}_2=&[\xi]^{2N+3}_{\frac{1}{4N}}+[\xi]^{2}_{\frac{1}{4N}}B_f+[\xi]^{2}_{\frac{1}{4N}}D_f+[\xi]_{\frac{1}{4N}}B_f^2+[\xi]_{\frac{1}{4N}}B_fC_f+[\xi]_{\frac{1}{4N}}B_fD_f+|\xi_{0}|^{4/3}B_f^{5/3}.\label{p2 formula}	
		\end{align}
		Similarly, by Lemma \ref{Weighted-Lp-Lp'}, Lemma \ref{est:partial f of R 2} and the corollary thereafter, we have 
		\begin{align*}
			\left\| \langle x \rangle^{-\sigma}B^{1/2} \hat{\mathcal{R}}_f \right\|_{L^4} \lesssim &  \int_{0}^{t}\left\|\langle x \rangle^{-\sigma}e^{-\mathrm{i}B (t-s)}B^{1/2}\partial_{\bar{f}}\mathcal{R}\right \|_{L^4}ds \\
			\lesssim & \sum_{d\in \{1,2,3,5\}}\int_{0}^{t}\min\{|t-s|^{-5/4},|t-s|^{-1/2}\}\|\langle x \rangle^{\sigma}B^{1/2}\partial_{\bar{f}}\mathcal{R}_{d}\|_{W^{1,4/3}}ds\\
			&+\int_{0}^{t}|t-s|^{-1/2}\|B^{1/2}\partial_{\bar{f}}\mathcal{R}_{4}\|_{W^{1,4/3}}ds\\
			\lesssim& \langle t \rangle^{-\frac{1}{2}}\mathcal{P}_3+\langle |\xi_0|^{4N}t \rangle^{-\frac{1}{2}}\mathcal{P}_4,
		\end{align*}
		where we choose 
		\begin{align}
			\mathcal{P}_3=&[\xi]^{2}_{\frac{1}{4N}}A_f+[\xi]^{2}_{\frac{1}{4N}}C_f+[\xi]_{\frac{1}{4N}}A_f^2+[\xi]_{\frac{1}{4N}}A_fC_f+[\xi]_{\frac{1}{4N}}A_fD_f+	|\xi_{0}|A_f^2, \label{p3 formula}\\	
			\mathcal{P}_4=&[\xi]^{2N+3}_{\frac{1}{4N}}+[\xi]^{2}_{\frac{1}{4N}}B_f+[\xi]^{2}_{\frac{1}{4N}}D_f+[\xi]_{\frac{1}{4N}}B_f^2+[\xi]_{\frac{1}{4N}}B_fC_f+[\xi]_{\frac{1}{4N}}B_fD_f+|\xi_{0}|^{-(2N-1)}B_f^{2}.\label{p4 formula}	
		\end{align}	
	\end{proof}	
	We finally come to the following error estimates:
	\begin{prop}\label{est:Rf Rxi}For $\sigma> \frac52$, the following estimates hold:
		\begin{align}
			&\|\langle x \rangle^{-\sigma}B^{-1/2}\mathcal{R}_{f}\|_{L^2}\lesssim   \langle t \rangle^{-\frac{2N+1}{4N}}\langle |\xi_0|^{4N}t \rangle^{-\frac{1}{4N}}\left(\|B^{-1/2}f(0)\|_X + |\xi_0|^{2N+1}+\mathcal{P}_1\right) + \langle |\xi_0|^{4N}t \rangle^{-\frac{2N+2}{4N}}\mathcal{P}_2,\\
			&|\mathcal{R}_{\xi}|\lesssim \langle t \rangle^{-\frac{4N+2}{4N}}\mathcal{P}_5+ \langle t \rangle^{-\frac{2N+1}{4N}}\langle|\xi_0|^{4N}t \rangle^{-\frac{2N+1}{4N}}\mathcal{P}_6 + \langle |\xi_0|^{4N}t \rangle^{-\frac{4N+2}{4N}}\mathcal{P}_7,
		\end{align}	
		where $\mathcal{P}_5, \mathcal{P}_6, \mathcal{P}_7$ are defined in \eqref{p5 formula}, \eqref{p6 formula} and \eqref{p7 formula} respectively.
	\end{prop}
	For the sake of convenience we first prove
	\begin{lem}\label{est:partial xi of R}
		\begin{align*}
			|\partial_{\bar{\xi}}\mathcal{R}|\lesssim  \langle t \rangle^{-\frac{4N+2}{4N}}\left([\xi]_{\frac{1}{4N}}A_f^2+A_f^3\right)
			+\langle |\xi_0|^{4N}t \rangle^{-\frac{4N+3}{4N}}\left([\xi]^{4N+3}_{\frac{1}{4N}}+[\xi]^{4N+2}_{\frac{1}{4N}}(A_f+B_f)+[\xi]_{\frac{1}{4N}}B_f^2+B_f^3\right).	
		\end{align*}
	\end{lem}
	\begin{proof}
		By using Leibnitz rule and expicit formula of $\mathcal{R} (iii.0-iii.5)$, we have 
		\begin{align*}
			&|\partial_{\bar{\xi}}\mathcal{R}_0|\lesssim|\xi|^{4N+3}\lesssim	\langle |\xi_0|^{4N}t \rangle^{-\frac{4N+3}{4N}}[\xi]^{4N+3}_{\frac{1}{4N}},\\
			&|\partial_{\bar{\xi}}\mathcal{R}_1|\lesssim|\xi|^{4N+2}\|B^{-1/2}f\|_{L^8}\lesssim \langle |\xi_0|^{4N}t \rangle^{-\frac{6N+3}{4N}}[\xi]^{4N+2}_{\frac{1}{4N}}(A_f+B_f),\\
			&|\partial_{\bar{\xi}}\mathcal{R}_2|\lesssim|\xi|\|B^{-1/2}f\|_{L^8}^2\lesssim \langle t \rangle^{-\frac{4N+2}{4N}}[\xi]_{\frac{1}{4N}}A_f^2+\langle |\xi_0|^{4N}t \rangle^{-\frac{4N+3}{4N}}[\xi]_{\frac{1}{4N}}B_f^2,\\
			&|\partial_{\bar{\xi}}\mathcal{R}_3|\lesssim\|B^{-1/2}f\|_{L^8}^3\lesssim \langle t \rangle^{-\frac{6N+3}{4N}}A_f^3+ \langle |\xi_0|^{4N}t \rangle^{-\frac{6N+3}{4N}}B_f^3,\\
			&|\partial_{\bar{\xi}}\mathcal{R}_4|\lesssim\|B^{-1/2}f\|_{L^8}^4\lesssim \langle t \rangle^{-\frac{8N+4}{4N}}A_f^4+ \langle |\xi_0|^{4N}t \rangle^{-\frac{8N+4}{4N}}B_f^4,\\
			&|\partial_{\bar{\xi}}\mathcal{R}_5|\lesssim|\xi|^{M^*}\lesssim \langle |\xi_0|^{4N}t \rangle^{-\frac{M^\star}{4N}}[\xi]^{M^\star}_{\frac{1}{4N}}.
		\end{align*}
		Hence, we get desired estimate of $\partial_{\bar{\xi}}\mathcal{R}$.
	\end{proof}
	\begin{proof}[Proof of Proposition \ref{est:Rf Rxi}]
		By H\"older's inequality,
		\begin{align*}
			&\left\|\langle x \rangle^{-\sigma}B^{-1/2}\mathcal{R}_{f}\right\|_{L^2}\\
			\lesssim &\left\|e^{-\mathrm{i}B t}B^{-1/2}f(0)\right\|_{L^8}+\left\|\int_{0}^{t}e^{-\mathrm{i}B (t-s)}B^{-1/2}\partial_{\bar{f}}\mathcal{R}ds\right\|_{L^8}\\
			&+\sum_{(\mu, \nu) \in M_1}\left|\eta^{\nu}(0)\bar{\eta}^{\mu}(0)\right|\left\|\langle x \rangle^{-\sigma}e^{-\mathrm{i}B t}(B-\omega (\nu-\mu)-\mathrm{i}0)^{-1}B^{-1/2}\bar{\Phi}_{\mu \nu}\right\|_{L^2}\\
			&+\sum_{(\mu, \nu) \in M_1}\int_{0}^{t}\left\|\langle x \rangle^{-\sigma}e^{-\mathrm{i}B (t-s)}(B-\omega (\nu-\mu)-\mathrm{i}0)^{-1}B^{-1/2}\bar{\Phi}_{\mu \nu}\right\|_{L^2}\left|\frac{d}{ds}(\eta^{\nu}\bar{\eta}^{\mu})\right|ds\\
			&+\sum_{(\mu, \nu) \in M \backslash M_1}\int_{0}^{t}\left\|\langle x \rangle^{-\sigma}e^{-\mathrm{i}B (t-s)}B^{-1/2}\bar{\Phi}_{\mu \nu}\right\|_{L^2}\left|\eta^{\nu}\bar{\eta}^{\mu}\right|ds\\
			:=& I+II+III+IV+V.
		\end{align*}
		In the proof of Proposition \ref{est:f}, we have showed that
		$$I=\|e^{-\mathrm{i}B t}B^{-1/2}f(0)\|_{L^8}\lesssim \langle t \rangle^{-9/8}\|B^{-1/2}f(0)\|_X.$$
		By Proposition \ref{thm:hat-Rf}, 
		\begin{align*}
			II=\left\|B^{-1/2}\hat{\mathcal{R}}_f\right\|_{L^8}  \lesssim \langle t \rangle^{-\frac{2N+1}{4N}}\mathcal{P}_1+
			\langle |\xi_0|^{4N}t \rangle^{-\frac{2N+2}{4N}}\mathcal{P}_2.
		\end{align*}
		By singular resolvents estimates Lemma \ref{lem-singular-resolvents}, we have 
		$$
		III\lesssim \langle t \rangle^{-6/5}|\xi_0|^{2N+1}.
		$$
		For the term $IV$, since $\left|\frac{d}{ds}(\eta^{\nu}\bar{\eta}^{\mu})\right|\lesssim |\eta|^{2N}|\dot{\eta}|$,
		and
		\begin{align*}
			|\dot{\eta}|=|\dot{\xi}  + \mathrm{i}\omega\xi |&\le |\partial_{\bar{\xi}}Z_0|+|\left\langle \partial_{\xi}G, f\right\rangle| +|\left\langle \partial_{\bar{\xi} }\bar{G}, \bar{f}\right\rangle|+|\partial_{\bar{\xi}}\mathcal{R}|\\
			&\lesssim  \langle t \rangle^{-\frac{4N+2}{4N}}\left([\xi]_{\frac{1}{4N}}A_f^2+A_f^3\right)+\langle|\xi_0|^{4N}t \rangle^{-\frac{3}{4N}}\left([\xi]^{3}_{\frac{1}{4N}}+[\xi]_{\frac{1}{4N}}B_f^2+B_f^3\right),				
		\end{align*}
		it holds that
		\begin{align*}
			IV\lesssim& \int_{0}^{t}\langle t-s \rangle^{-6/5}\langle s \rangle^{-\frac{4N+2}{4N}}[\xi]^{2N}_{\frac{1}{4N}}\left([\xi]_{\frac{1}{4N}}A_f^2+A_f^3\right)ds\\
			&+\int_{0}^{t}\langle t-s \rangle^{-6/5}\langle|\xi_0|^{4N}s \rangle^{-\frac{2N+3}{4N}}[\xi]^{2N}_{\frac{1}{4N}}\left([\xi]^{3}_{\frac{1}{4N}}+[\xi]_{\frac{1}{4N}}B_f^2+B_f^3\right)ds\\
			\lesssim& \langle t \rangle^{-\frac{4N+2}{4N}}[\xi]^{2N}_{\frac{1}{4N}}\left([\xi]_{\frac{1}{4N}}A_f^2+A_f^3\right)+\langle|\xi_0|^{4N}t \rangle^{-\frac{2N+3}{4N}}[\xi]^{2N}_{\frac{1}{4N}}\left([\xi]^{3}_{\frac{1}{4N}}+[\xi]_{\frac{1}{4N}}B_f^2+B_f^3\right)	.
		\end{align*}
		Similarly, 
		\begin{align*}
			V\lesssim& \int_{0}^{t}\langle t-s \rangle^{-6/5}\langle|\xi_0|^{4N}s \rangle^{-\frac{2N+3}{4N}}ds[\xi]^{2N+3}_{\frac{1}{4N}}\\
			\lesssim& \langle|\xi_0|^{4N}t \rangle^{-\frac{2N+3}{4N}}[\xi]^{2N+3}_{\frac{1}{4N}}.
		\end{align*}
		Combining the above estimates and recall the formula of $\mathcal{P}_1, \mathcal{P}_2$, we get desired estimates for $\mathcal{R}_{f}.$ 
		
		Then using formula \eqref{eq:Rxi} of $\mathcal{R}_{\xi}$, also combining the estimates of $\mathcal{R}_{f}$ and Lemma \ref{est:partial xi of R}, we get estimates of $\mathcal{R}_{\xi}$:
		\begin{align*}
			|\mathcal{R}_{\xi}|\lesssim& \langle |\xi_0|^{4N}t \rangle^{-\frac{4N+3}{4N}}[\xi]^{4N+3}_{\frac{1}{4N}}+\langle|\xi_0|^{4N}t \rangle^{-\frac{1}{2}}[\xi]^{2N}_{\frac{1}{4N}}\|\langle x \rangle^{-\sigma}B^{-1/2}\mathcal{R}_{f}\|_{L^2}+|\partial_{\bar{\xi}}\mathcal{R}|\\
			\lesssim& \langle t \rangle^{-\frac{4N+2}{4N}}\mathcal{P}_5+ \langle t \rangle^{-\frac{2N+1}{4N}}\langle|\xi_0|^{4N}t \rangle^{-\frac{2N+1}{4N}}\mathcal{P}_6 + \langle |\xi_0|^{4N}t \rangle^{-\frac{4N+2}{4N}}\mathcal{P}_7,
		\end{align*}
		where 
		\begin{align}
			\mathcal{P}_5&=[\xi]_{\frac{1}{4N}}A_f^2+A_f^3,\label{p5 formula}\\
			\mathcal{P}_6&=[\xi]^{2N}_{\frac{1}{4N}}\left(\|B^{-1/2}f(0)\|_X + |\xi_0|^{2N+1}+\mathcal{P}_1\right)  ,\label{p6 formula}\\
			\mathcal{P}_7&=[\xi]^{2N}_{\frac{1}{4N}}\mathcal{P}_2+[\xi]^{4N+2}_{\frac{1}{4N}}(A_f+B_f)+[\xi]_{\frac{1}{4N}}B_f^2+B_f^3  .\label{p7 formula}
		\end{align}
	\end{proof}

	\section{Proof of Theorem \ref{main-result}}\label{Sec-proof of main result}
	Now we shall prove Theorem \ref{main-result} by bootstrap arguments. Denote the original variables to be 
	$(\xi',f') = \mathcal{T}_{2N}(\xi,f)$, then by \eqref{assumption-data} the initial data $(\xi'_0, f'(0))$ satisfies
	\begin{align*}
		\|B^{-1/2}f'(0)\|_X\lesssim |\xi'_{0}|\lesssim \epsilon.
	\end{align*}
	By Theorem \ref{thm:nft} (ii), it is equivalent to
	\begin{align*}
		\|B^{-1/2}f(0)\|_X\lesssim |\xi_{0}|\lesssim \epsilon.
	\end{align*}
	Then by energy conservation, we have
	\begin{equation}\label{energy esti}
		|\xi(t)|+\|B^{1/2}f\|_{L^2} \lesssim |\xi(0)|+\|B^{1/2}f(0)\|_{L^2} \le (1+C_0)|\xi_0|.
	\end{equation}
	
	\begin{prop}\label{apriori est}
		Assume $[\xi]_{\frac{1}{4N}}(T)\lesssim |\xi_0|$, then \eqref{assumption f1} and \eqref{assumption f2} hold with
		\begin{align*}
			A_f(T)&\lesssim |\xi_{0}| ,\\
			B_f(T)&\lesssim |\xi_{0}|^{2N+1} ,\\
			C_f(T)&\lesssim |\xi_{0}| ,\\
			D_f(T)&\lesssim |\xi_{0}|^{2N+1} ,\\
		\end{align*}
		consequently,
		\begin{align*}
			|\mathcal{R}_{\xi}|\lesssim \langle t \rangle^{-\frac{4N+2}{4N}}|\xi_{0}|^3+ \langle t \rangle^{-\frac{2N+1}{4N}}\langle|\xi_0|^{4N}t \rangle^{-\frac{2N+1}{4N}}|\xi_{0}|^{2N+1} + \langle |\xi_0|^{4N}t \rangle^{-\frac{4N+2}{4N}}|\xi_{0}|^{4N+3}.
		\end{align*}	
	\end{prop}
	\begin{proof}
		The proof consists of a bootstrap argument using Proposition \ref{est:f} and \ref{thm:hat-Rf}. First, by \eqref{energy esti}, the result \eqref{assumption f1} and \eqref{assumption f2} hold for small $t<1$. Then, assume this is true for $0<t<t^*<T$ with
		\begin{align*}
			A_f(T)&\le K |\xi_{0}| ,\\
			B_f(T)&\le K |\xi_{0}|^{2N+1} ,\\
			C_f(T)&\le K |\xi_{0}| ,\\
			D_f(T)&\le K |\xi_{0}|^{2N+1} ,\\
		\end{align*}
		where $K$ is a large constant to be chosen. By $[\xi]_{\frac{1}{4N}}(T)\lesssim  |\xi_0|$ and Proposition \ref{thm:hat-Rf}, \eqref{assumption RfL8} and \eqref{assumption RfL4} hold for $0<t<t^*$ with
		\begin{align*}
			\mathcal{P}_1&\lesssim |\xi_{0}|^2A_f+|\xi_{0}|C_f+|\xi_{0}|A_f^2+ |\xi_{0}|A_f C_f+|\xi_{0}|A_f D_f+  |\xi_{0}|^{4/3}A_f^{5/3}\lesssim K^2|\xi_{0}|^2,\\
			\mathcal{P}_2&\lesssim |\xi_{0}|^{2N+3}+ |\xi_{0}|^{2}B_f+ |\xi_{0}|^{2}D_f+|\xi_{0}|B_f^2+|\xi_{0}|B_fC_f+|\xi_{0}|B_fD_f+|\xi_{0}|^{4/3}B_f^{5/3}\lesssim K^2|\xi_{0}|^{2N+3},\\
			\mathcal{P}_3&\lesssim |\xi_{0}|^2A_f+ |\xi_{0}|^2C_f+|\xi_{0}|A_f^2+ |\xi_{0}|A_f C_f+|\xi_{0}|A_f D_f+ |\xi_{0}|A_f^{2}\lesssim K^2|\xi_{0}|^3,\\
			\mathcal{P}_4&\lesssim |\xi_{0}|^{2N+3}+|\xi_{0}|^{2}B_f+|\xi_{0}|^{2}D_f+|\xi_{0}|B_f^2+|\xi_{0}|B_fC_f+|\xi_{0}|B_fD_f+|\xi_{0}|^{-(2N-1)}B_f^{2}\lesssim K^2|\xi_{0}|^{2N+3}.
		\end{align*}
		By Proposition \ref{est:f}, we have 
		\begin{align*}
			&\|B^{-1/2}f\|_{L^8}\lesssim \langle t \rangle^{-\frac{2N+1}{4N}}\Big(|\xi_{0}|+K^2|\xi_{0}|^2\Big)
			+\langle |\xi_0|^{4N}t \rangle^{-\frac{2N+1}{4N}}\Big(|\xi_{0}|^{2N+1}+K^2|\xi_{0}|^{2N+3}\Big)
			,\\
			&\|\langle x \rangle^{-\sigma}B^{1/2}f\|_{L^4}\lesssim \langle t \rangle^{-\frac{1}{2}}\Big(|\xi_{0}|+K^2|\xi_{0}|^3\Big)+\langle |\xi_0|^{4N}t \rangle^{-\frac{1}{2}}\Big(|\xi_{0}|^{2N+1}+K^2|\xi_{0}|^{2N+3}\Big).
		\end{align*}
		By choosing $K$ large and $|\xi_0|$ small we complete the bootstrap argument. In addition, we have
		\begin{align}\label{p formula}
			\mathcal{P}_5&=[\xi]_{\frac{1}{4N}}A_f^2+A_f^3\lesssim |\xi_{0}|^3,\\
			\mathcal{P}_6&=[\xi]^{2N}_{\frac{1}{4N}}\left(\|B^{-1/2}f(0)\|_X + |\xi_0|^{2N+1}+\mathcal{P}_1\right)\lesssim |\xi_{0}|^{2N+1} ,\\
			\mathcal{P}_7&=[\xi]^{2N}_{\frac{1}{4N}}\mathcal{P}_2+[\xi]^{4N+2}_{\frac{1}{4N}}(A_f+B_f)+[\xi]_{\frac{1}{4N}}B_f^2+B_f^3\lesssim |\xi_{0}|^{4N+3},
		\end{align}
		by Proposition \ref{est:Rf Rxi} we get estimates of $|\mathcal{R}_{\xi}|$.
	\end{proof}

	\begin{proof}[Proof of Theorem \ref{main-result}]
		We use bootstrap argument again to show that
		\begin{align}\label{eq:xi-apt}
			[\xi]_{\frac{1}{4N}}(T)\le 8C^*|\xi_0|
		\end{align}
		holds for all $T>0$, where $C^*= [1+ (4N\gamma)^{-1}]^\frac{1}{4N}(1+C_0)$ is an absolute constant. Note that by energy conservation, \eqref{eq:xi-apt} holds for $T\le |\xi_0|^{-4N(1-\delta_{0})}$. For $T\ge |\xi_0|^{-4N(1-\delta_{0})},$ we assume that
		\begin{align}\label{assumption xi}
			[\xi]_{\frac{1}{4N}}(T^*)\le 8C^* |\xi_0|
		\end{align}
		for some $|\xi_0|^{-4N(1-\delta_{0})}\le T^*<T$. Then, by Proposition \ref{apriori est}, we have 	
		$$
		|\mathcal{R}_{\xi}|\lesssim \langle |\xi_0|^{4N}t \rangle^{-\frac{4N+2}{4N}}\left(|\xi_0|^{(4N+2)(1-\delta_{0})}+|\xi_0|^{(4N+2)(1-\delta_{0}/2)}+|\xi_0|^{4N+3}\right)
		$$
		for $0\le t\le T^*$. Choosing $\delta_{0}<\frac{1}{4N+2}$, we have 
		$
		|\mathcal{R}_{\xi}|\lesssim \langle |\xi_0|^{4N}t \rangle^{-\frac{4N+2}{4N}}|\xi_0|^{4N+1+\delta_1}
		$ for some $\delta_1>0.$
		Hence, we can apply Theorem \ref{ode} at $t_0= |\xi_0|^{-4N(1-\delta_{0})}$ with $Q_0 \approx |\xi_0|^{4N+1+\delta_1}.$ Note that 
		\begin{align*}
			|r(t_0) - r(0)| &\le \int_{0}^{t_0}\Big(4N\gamma r^2(t)+4N	|\mathcal{R}_{\xi}|r^{\frac{4N-1}{4N}}(t)\Big)dt\\
			&\le C\int_{0}^{t_0}\Big(|\xi_0|^{8N}+|\xi_0|^{4N+1+\delta_1}|\xi_0|^{4N-1}\Big)dt\\
			&\le C|\xi_0|^{4N(1+\delta_{0})}\\
			&\le \frac{1}{2}r(0),
		\end{align*}
		we have $\frac12 |\xi_{0}|^{4N} \le |\xi(t_0)|^{4N}\le \frac32|\xi_{0}|^{4N}.$ Therefore, we get
		\begin{align*}
			|\xi(t)|^{4N}\le & 4(1+4N\gamma|\xi_0|^{4N} t)^{-1}\left(|\xi_0|^{4N}+\frac{C(C^*)|\xi_0|^{4N+\delta_1/2}}{(1+4N\gamma|\xi_0|^{4N} t)^{\delta/2}}\right)\\
			\le& 5\left(C^*\right)^{4N}(1+|\xi_0|^{4N} t)^{-1}|\xi_0|^{4N}.	
		\end{align*}
		The last inequality holds by using
		$$(1+4N\gamma|\xi_0|^{4N} t)^{-1}\le \left(C^*\right)^{4N} (1+|\xi_0|^{4N} t)^{-1} $$
		and choosing $|\xi_0|$ to be sufficiently small such that $C(C^*)|\xi_0|^{\delta_1/2}\le 1$.

		Next, we shall prove the lower bound of $|\xi|$. We have proved that 
		$
		|\xi(t)|^{4N}\approx |\xi_0|^{4N}\approx(1+|\xi_0|^{4N} t)^{-1}|\xi_0|^{4N}
		$
		for $0\le T\le |\xi_0|^{-4N(1-\delta_{0})}$. For $T\ge |\xi_0|^{-4N(1-\delta_{0})},$ by Proposition \ref{apriori est} we have 
		$$
		|\mathcal{R}_{\xi}|\lesssim \langle |\xi_0|^{4N}t \rangle^{-\frac{4N+2}{4N}}|\xi_0|^{4N+1+\delta_1},
		$$
		thus using Lemma \ref{ode} we also obtain $|\xi(t)|^{4N}\gtrsim (1+|\xi_0|^{4N} t)^{-1}|\xi_0|^{4N}$. Combining above estimates we obtain 
		$$ |\xi(t)|\approx \frac{|\xi_0|}{(1+|\xi_0|^{4N} t)^{\frac{1}{4N}}}.$$
		By Proposition \ref{apriori est}, we also have 
		$$\|B^{-1/2}f\|_{L^8}\lesssim \langle t \rangle^{-\frac{2N+1}{4N}}|\xi_0|
		+\langle |\xi_0|^{4N}t \rangle^{-\frac{2N+1}{4N}}|\xi_0|^{2N+1}.$$ 
		By the property of normal form transformation, for the original variables
		$(\xi',f') = \mathcal{T}_{2N}(\xi,f)$,
		we have 
		$$ |\xi'(t)|\approx \frac{|\xi'_0|}{(1+|\xi'_0|^{4N} t)^{\frac{1}{4N}}}$$
		and
		$$\|B^{-1/2}f'\|_{L^8}\lesssim \langle t \rangle^{-\frac{2N+1}{4N}}|\xi'_0|
		+\langle |\xi_0'|^{4N}t \rangle^{-\frac{2N+1}{4N}}|\xi'_0|^{3}.$$
		Write $\xi'(t)=\rho(t)e^{-\mathrm{i}(\omega t+\theta(t))}$, then
		$q(t)=\sqrt{\frac{2}{\omega}} \operatorname{Re} \xi(t)=\sqrt{\frac{2}{\omega}}\rho(t)\cos(\omega t+\theta(t))$, and
		$$u(t,x) = R(t) \cos(\omega t+\theta(t))\varphi(x) +\eta(t,x)$$ with $R(t)=\sqrt{\frac{2}{\omega}}\rho(t).$
		Hence, the above estimates are equivalent to
		$$ |R(t)|\approx \frac{|R(0)|}{(1+|R(0)|^{4N} t)^{\frac{1}{4N}}}$$
		and
		$$\|\eta(t)\|_{L^8}\lesssim \frac{|R(0)|}{(1+ t)^{\frac{2N+1}{4N}}}+ \frac{|R(0)|^{3}}{(1+|R(0)|^{4N} t)^{\frac{3}{4N}}}.$$
		In addition, since 
		$$\dot{\xi}'+\mathrm{i}\omega\xi'=\dot{\rho}e^{-\mathrm{i}(\omega t+\theta)}-\mathrm{i}\dot{\theta}\xi',
		$$
		where $|\dot{\xi}'+\mathrm{i}\omega\xi'|=\mathcal{O}(|t|^{-\frac{3}{4N}}), |\dot{\rho}|=\mathcal{O}(|t|^{-\frac{3}{4N}})$, we get $|\dot{\theta}|=\mathcal{O}(|t|^{-\frac{1}{2N}})$, hence $\theta(t) = \mathcal{O}(|t|^{1-\frac{1}{2N}}),$	which completes the proof of Theorem \ref{main-result}.
	\end{proof}
	
	\section*{Acknowledgment}
	We would like to thank H. Jia for introducing the problem to us and for the useful discussions. The authors were in part supported by NSFC (Grant No. 11725102), National Support Program for Young Top-Notch Talents, and Shanghai Science and Technology Program (Project No. 21JC1400600 and No. 19JC1420101). 
	
	\appendix
	\section{Convolution Estimates}
	The following lemma deals with the decay rate of the convolution of two functions, which is used frequently in the asymptotic analysis in Section \ref{Sec-aymptotic behavior} and Section \ref{Sec-error estimates}.
	\begin{lem}\label{convo1}
		Let $0<\delta<1$ and $0< f(t) \lesssim  \min\{|t|^{-(1+\delta)},|t|^{-(1-\delta)}\}$, then
		\begin{equation}
			\int_{0}^{t}f(t-s)\langle |\xi_0|^{4N}s\rangle^{-\alpha}ds\lesssim \langle |\xi_0|^{4N}t\rangle^{-\alpha}, \quad  \forall \  0\le \alpha\le 1+\delta.
		\end{equation}
		
	\end{lem}
	\begin{proof} Note that $f$ is integrable on the whole line, thus it suffices to consider $t\ge 2|\xi_0|^{-4N}$. Then we have
		\begin{align}
			\int_{0}^{t}f(t-s)\langle |\xi_0|^{4N}s\rangle^{-\alpha}ds\le& \int_{0}^{|\xi_0|^{-4N}}f(t-s)\langle |\xi_0|^{4N}s\rangle^{-\alpha}ds+\int_{|\xi_0|^{-4N}}^{t/2}f(t-s)\langle |\xi_0|^{4N}s\rangle^{-\alpha}ds\nonumber\\
			&+\int_{t/2}^{t}f(t-s)\langle |\xi_0|^{4N}s\rangle^{-\alpha}ds\nonumber\\
			\triangleq &I+II+III,    	\nonumber
		\end{align}
		where 
		\begin{align*}
			I&\lesssim t^{-(1+\delta)}|\xi_0|^{-4N} \lesssim  \langle |\xi_0|^{4N}t\rangle^{-(1+\delta)},\\
			II&\lesssim  t^{-(1+\delta)}\int_{|\xi_0|^{-4N}}^{t/2}\langle |\xi_0|^{4N}s\rangle^{-\alpha}ds \lesssim  \begin{cases}
				t^{-(1+\delta)}|\xi_0|^{-4N\alpha}t^{1-\alpha} \lesssim  (|\xi_0|^{4N}t)^{-\alpha} &  if \quad\alpha<1,\\ t^{-(1+\delta)}|\xi_0|^{-4N\alpha}|\xi_0|^{4N(\alpha-1)}\lesssim (|\xi_0|^{4N}t)^{-(1+\delta)} & if \quad \alpha>1,
			\end{cases} \\
			III&\lesssim   \langle |\xi_0|^{4N}t\rangle^{-\alpha}\int_{t/2}^{t}f(t-s)ds	\lesssim \langle |\xi_0|^{4N}t\rangle^{-\alpha}.
		\end{align*} 
		Combining the above estimates, we obtain the desired decay estimate.
	\end{proof}
	Similarly, we can also prove the following lemma:
	\begin{lem}\label{convo2}
		For $\alpha>1$, we have 
		\begin{equation}
			\int_{0}^{t}(t-s)^{-1/2}\langle |\xi_0|^{4N}s\rangle^{-\alpha}ds\lesssim |\xi_0|^{-2N}\langle |\xi_0|^{4N}t\rangle^{-1/2}.
		\end{equation}
	\end{lem}
	%%      ---------------------------------------------------------------------
	%%      --------------------------- BIBLIOGRAPHY ----------------------------
	%%      ---------------------------------------------------------------------
	
	\frenchspacing
	\bibliographystyle{plain}
	
	%\end{CJK*}
\end{document}